\documentclass[12pt,a4paper]{paper}

\usepackage{amssymb,amsmath, amsthm, amsfonts}
\usepackage{enumerate}
\usepackage{geometry}
\usepackage{tikz}
\usepackage{footmisc}

\newtheorem{definition}{\vspace{1mm}Definition}[section]
\newtheorem{remark}[definition]{\vspace{1mm}Remark}
\newtheorem{lemma}[definition]{\vspace{1mm}Lemma}
\newtheorem{theorem}[definition]{\vspace{1mm}Theorem}
\newtheorem{corollary}[definition]{\vspace{1mm}Corollary}
\newtheorem{proposition}[definition]{\vspace{1mm}Proposition}
\newtheorem{notation}[definition]{\sc Notation}
\newtheorem{exam}[definition]{\vspace{1mm}Example}

\newcommand{\axD}{\textsf{(D)}}
\newcommand{\axK}{\textsf{(K)}}
\newcommand{\axKu}{\textsf{(K1)}}
\newcommand{\axKd}{\textsf{(K2)}}
\newcommand{\axMu}{\textsf{(M1)}}
\newcommand{\axMd}{\textsf{(M2)}}
\newcommand{\axMt}{\textsf{(M3)}}
\newcommand{\axMc}{\textsf{(M4)}}
\newcommand{\axDNu}{\textsf{(DN1)}}
\newcommand{\axDNd}{\textsf{(DN2)}}
\newcommand{\axT}{\textsf{(T)}}
\newcommand{\axKp}{\textsf{(K')}}
\newcommand{\axKup}{\textsf{(K1')}}
\newcommand{\axKdp}{\textsf{(K2')}}
\newcommand{\axMtp}{\textsf{(M3')}}
\newcommand{\axMcp}{\textsf{(M4')}}
\newcommand{\axIu}{\textsf{(I1)}}
\newcommand{\axId}{\textsf{(I2)}}
\newcommand{\axfour}{\textsf{(4)}}
\newcommand{\axfive}{\textsf{(5)}}
\newcommand{\MP}{\textsf{MP}}
\newcommand{\axKpp}{\textsf{(K'')}}
\newcommand{\axKupp}{\textsf{(K1'')}}
\newcommand{\axKdpp}{\textsf{(K2'')}}
\newcommand{\axMtpp}{\textsf{(M3'')}}
\newcommand{\axMcpp}{\textsf{(M4'')}}
\newcommand{\axIup}{\textsf{(I1')}}
\newcommand{\axIdp}{\textsf{(I2')}}
\newcommand{\axIupp}{\textsf{(I1'')}}
\newcommand{\axIdpp}{\textsf{(I2'')}}
\newcommand{\Kdet}{\textsf{(Kdet)}}

\title{Modal Logic With Non-deterministic Semantics: Part I - Propositional Case}
\author{Marcelo E.~Coniglio$^{1}$, Luis Fari\~nas del Cerro$^{2}$ and Newton M. Peron$^{3}$\\ [2mm] %
{\small $^1$Institute of Philosophy and the Humanities (IFCH) and}\\[-2mm]
{\small Centre for Logic, Epistemology and The History of Science (CLE),}\\[-2mm]
{\small University of Campinas (UNICAMP), Campinas, SP, Brazil.}\\[-2mm]
{\small E-mail:~{coniglio@cle.unicamp.br}}\\[1mm]
{\small $^2$IRIT, Université de Toulouse,  CNRS, France. }\\[-2mm]
{\small E-mail:~{farinas@irit.fr}}\\[1mm]
{\small $^3$Academic Unit of Philosophy and Human Sciences,}\\[-2mm]
{\small Federal University of Southern Frontier (UFFS), Chapec\'o, SC, Brazil.}\\[-2mm]
{\small E-mail:~{newton.peron@uffs.edu.br}}}

\begin{document}

\maketitle

\begin{abstract}

In 1988, Ivlev proposed  four-valued non-deterministic semantics for modal logics in which the alethic \axT\ axiom holds good. Unfortunately, no completeness was proved. In previous work, we proved completeness for some Ivlev systems and extended his hierarchy, proposing weaker six-valued systems in which the \axT\ axiom was replaced by the deontic \axD\ axiom. Here, we eliminate both axioms, proposing even weaker systems with eight values. Besides, we prove completeness for those new systems.
It is natural to ask if a characterization by finite ordinary (deterministic) logical matrices would be possible for all those systems. We will show that finite deterministic matrices do not characterize any of them.

\end{abstract}

\


\section*{Introduction}

Non-deterministic matrices (Nmatrices) is an useful alternative  in order to give a satisfactory semantic account of certain logics that are not characterizable by finite deterministic matrices. In deterministic matrices, the truth-value of a complex formula is uniquely
determined by the truth-values of its subformulas by means of truth-funcions. In a non-deterministic matrix, the truth-value of a complex formula can be chosen non-deterministically out of some non-empty set of options\footnote{Acording to \cite{avr:zam:11}.}.

It is worth noting that most part of standard modal systems can not be characterizable by finite deterministic matrices.\footnote{In fact, Dugundji proved in \cite{dug:40} this result for Lewis' hierarchy {\bf S1}-{\bf S5}, but this result can be easily extended to most part of standard modal systems. See \cite{con:per:14}.}  Taking this into account, it is natural to think  if finite non-deterministic matrices can be considered as a useful approach to modal logic. Actually, Ivlev had proposed  in the eighties a semantics of four-valued Nmatrices\footnote{Ivlev called them {\em quasi-matrices}. See \cite{ivl:88}.}  for implication and modal operators, in order to characterize a hierarchy of  weak modal logics without necessitation rule\footnote{Necessitation Rule is a rule that infers, from a given theorem, that it is necessarily a theorem. Namely: if $\vdash \alpha$ then $\vdash \Box \alpha$.}.

Ivlev assumed that  everything that is necessarily true is actually true. In order to considerate systems without this assumption, we extended Ivlev's hiearchy\footnote{See in \cite{con:cer:per:15}}, by adding two new truth-values. In section 1, we will discuss how those four and six truth-values  could be interpreted in terms of modal concepts. In the first case, we will offer an alethic perspective; in the second one, we will suggest a deontic interpretation. 

In a deontic context, we assume that what is necessarily (or obliged) is, at least, possible (or permitted). In this paper, we will propose weaker systems in which this principle does not hold. So, we can have modal contradictory situations, that means, propositions that are necessarily true but impossible. This requires eight truth-values. We will offer an epistemic interpretation of them in Section~\ref{values}.

In Section~\ref{operators}, we will argue why non-deterministic implication and modal operators are required considering how modal concepts work in the natural language. We will start with four values, going towards six and eight values. After that, we will prove in Section~\ref{Km-compl} completeness for some modal systems with eight truth-values.

Non-deterministic propositional modal logics presented here clearly constitute a hierarchy. The point here, as shown in Section~\ref{DAT}, it is not only that one system is included in another one. Moreover, it will be shown that any system of the hierarchy can recover all the inferences of the stronger ones by means of Derivability Adjustment Theorems.

In Ivlev's semantics, four kinds of implications are considered: those who have three, two, one or zero cases of non-determinism. However, no completeness result is presented. In another work, we proved completeness for some systems whose implication has three cases of non-determinism\footnote{In \cite{con:cer:per:15}, we proved those completeness using Ivlev's rule (DN). This rule, unfortunately, is not semantically sound, as it was remarked and fixed in \cite{omo:sku:16}. In \cite{con:cer:per:17}, we replaced the rule by the axioms (DN1) and (DN2), also fixing the problem.}. In section \ref{sectDelta}, we will prove completeness with respect to deterministic Ivlev's implication, showing some relation between this modal systems, four-valued G\"odel  logic, four-valued \L ukasiewicz logic and Monteiro-Baaz's  $\Delta$ operator.

It is also natural to ask whether the use of  finite non-deterministic matrices is essential  for dealing with those hierarchies, or if a characterization by finite ordinary deterministic matrices would be  possible. In section \ref{SectDugun}, we will show that finite deterministic matrices do not characterize any of these systems that are complete with respect to finite Nmatrices.

\section{Interpreting modal concepts as truth-values} \label{values}

Consider the modal concepts of necessarily true, possibly true and actually true. Their negations are necessarily false (or impossible), possibly false and actually false,  respectively. If we consider contingent as being, at the same time, possible true and possible false, this permits us to define eight truth-values:

\begin{itemize}
  \item[] \ $T^+$: necessarily, possibly and actually true;
  \item[] \  $C^+$: contingently and actually true;
  \item[] \  $F^+$: impossible, possibly false but actually true;
  \item[] \  $I^+$: necessary true, impossible and actually true;
  \item[] \  $T^-$: necessarily and possibly true but actually false;
  \item[] \  $C^-$: contingently and actually false;
  \item[] \  $F^-$: impossible, possibly false and actually false;
  \item[] \  $I^-$: necessarily true, impossible and actually false.\footnote{These truth-values can be formalized in a modal language (assuming, as usual, the equivalences $\neg\Box p \equiv \Diamond \neg p$ and $\neg\Diamond p \equiv \Box \neg p$)  as follows: $T^+$: $\Box p \wedge \Diamond p \wedge p$; $C^+$:  $\neg\Box p \wedge \Diamond p \wedge p$; $F^+$:  $\Box \neg p \wedge \Diamond \neg p \wedge p$; $I^+$: $\Box p \wedge \neg\Diamond p \wedge p$; $T^-$: $\Box p \wedge \Diamond p \wedge \neg p$; $C^-$:  $\neg\Box p \wedge \Diamond p \wedge \neg p$; $F^-$:  $\Box \neg p \wedge \Diamond \neg p \wedge \neg p$; $I^-$: $\Box p \wedge \neg\Diamond p \wedge \neg p$. Modal languages will be discussed in Section~\ref{operators}.}
\end{itemize}

In order to better understand the meaning of those truth-values, we can, first of all, grouping them into two subsets:

\begin{itemize}
  \item[] \  $+ = \{T^+,C^+,F^+,I^+ \}$ (actually true);
  \item[] \  $- = \{T^-,C^-,F^-,I^- \}$ (actually false).
\end{itemize}

When a proposition receives a truth-value in $+$, we understand that it is actually or factually the case or, in Kripkean terms, it is true ``in the real world''. Otherwise, it receives a truth-value in $-$; in that case, we interpret that the proposition is not actually or factually  the case, or it is false ``in the real world''.
The  truth-values in ``actually true'' are designated but the truth-values in ``actually false'' are not. All the modal systems here considered are extensions of Propositional Classical Logic - {\bf PC}. That means: all the tautologies are actually true and all the contradictions are actually false.

Kripke semantics for propositional modal logic also respects  classical truth-tables, but there are models when the accessibility relation is empty in which all the formulas are necessarily true but also impossible\footnote{Those models were not, in fact, proposed by Kripke according to  \cite[p.  49]{hug:cre:96}. Actually, Kripke assumes in \cite{kri:63} that the accessibility relation for normal system is, at least, reflexive. Segerberg  proposes an equivalent axiomatic in 1971, defending that, against ``followers of Kripke's terminology'', normal modal logics should cover those models whose the accessibility relation is empty, see \cite[p. 12]{seg:71}.\label{K}}. Analogously, the values $I^+$ and $I^-$ expresses those modal contradictory scenarios.

Kripke argues that\footnote{ In \cite[p. 34]{kri:81}.} his semantics, differently from Kant, distinguishes the concepts of necessary and \emph{a priori}. In fact, for Kant, sentences (or ``judges'', in Kantian terms)  are necessary if and only if we can think of them \emph{a priori}, that is, not empirically\footnote{In the first reading of \emph{Critique of Pure Reason}, the reciprocate does not seem to be true. Indeed, a judge is \emph{a priori} if and only if it is necessarily or strictly universal. But necessity and strictly universality belong together inseparably \cite[(B2-B3)]{kant:98}. Finally, Kant affirms in \emph{Prolegomena} that necessity and cognition taken for \emph{a priori} are the same \cite[(4:277)]{kant:92} .}. But Kripke claims that ``necessary'' has a metaphysical meaning captured by his semantics: something is necessarily true if and only if it is true in all possible worlds. Therefore, in Kripkean terms, we can know something empirically that is true in all the possible worlds, and so, that would be necessary \emph{a posteriori}. Reciprocally, we can know something \emph{a priori} but there could be a world in which that would be false, so it would be contingent \emph{a priori}.

The non-deterministic semantics proposed in the present paper, in turn, intends to recuperate the identification between being necessary and being \emph{a priori} that most part of philosophers does, as even Kripke admits. In fact, this identification is not only in accordance with Kant's point of view. A correlation between the notions of being necessary and being \emph{a priori} is found in others important philosophers, for instance, Descartes\footnote{Descartes  does not use the term \emph{a priori} in his \emph{Meditations on First Philosophy}, see \cite{des:11}. However, he  separates the study of composite things, like medicine and physics, from simpler things, like geometry and arithmetic (AT 20). Gueroult, a Descartes' commentator, arguments that, on First Meditation, Descartes presents some natural reasons to doubt about the composite things, instead of the simpler ones, since they are the necessary conditions of all possible representations, that is, for those things that exist in nature, see \cite[p. 17]{gue:68}. Those simpler things are  intuited ``with the eye of the mind'' (AT 36). Besides, mathematical truths are ``the most certain of all'' (AT 65) and in mathematics there is ``true and certain knowledge of it [scientia]'' (AT 70). The Cartesian characterizations of arithmetic and geometry knowledge is very near to the Kantian notion of \emph{a priori}. Because of those things clearly and distinctly intuited are necessarily true (AT 70), we can infer that Descartes defends, agreeing with Kant, that \emph{a priority} and necessity are attached concepts. This approximation must be done, of course, with some caveats. For instance, both philosophers could disagree with respect to the role of the imagination in \emph{a priori} knowledge, compare, in particular, \emph{Meditations} (AT 72-73) and \emph{Critique} (B 103-104).  } and more contemporary, 
Russell\footnote{In \cite[p. 64-65]{russ:10}, Russell defends that the term \emph{necessary} should not be applied to propositions as ordinary traditional philosophy does; instead it should be applied to propositional functions: if a propositional function is always true, it is \emph{necessary} true (for ex. ``if $x$ is a man, $x$ is mortal''), if it is sometimes true, it is \emph{possible} (for ex. ``$x$ is a men'') and if it is never true, it is \emph{impossible} (for ex. ``$x$ is a unicorn'').  Further \cite[p. 75]{russ:10}, he says that we know logical propositions \emph{a priori} and one of his examples is the sentence ``If all $a$'s are $b$'s and $x$ is an $a$, then $x$ is a $b$''. That means that we know \emph{a priori} all logical theorems. For soundness, any logical theorem is always true. In other words, what we know \emph{a priori} is necessarily true.\label{russell}}.

To endorse that perspective, the eight values should be grouped as follows: 

\begin{itemize}
  \item[] \ $T = \{T^+,T^-\}$ ( \emph{a priori} true)
  \item[] \  $C = \{C^+,C^-\}$ ( \emph{a posteriori} true or false)
  \item[] \  $F = \{F^+,F^-\}$ ( \emph{a priori} false)
  \item[] \  $I = \{I^+,I^-\}$ (\emph{a priori} true and false )
\end{itemize}

In conceptual terms, let $\mathbb{A}$, 
$\mathbb{N}$ and $\mathbb{P}$ represent the concepts of \emph{actual}, \emph{necessary} and \emph{possible}, respectively. The Venn-diagram below shows the relationship between the truth-values and those concepts:

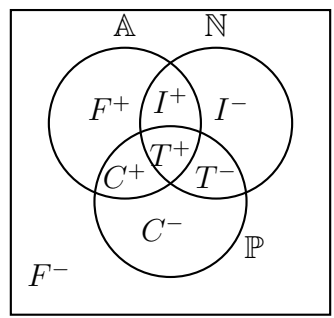
\begin{center}
\begin{tikzpicture}[thick] 
\draw (2.7,-2.54) rectangle (-1.5,1.5) node[below right] {}; 
\draw (0,0) circle (1) node[above,shift={(0,1)}] {$\mathbb{A}$};
\draw (1.2,0) circle (1) node[above,shift={(0,1)}] {$\mathbb{N}$};
\draw (.6,-1.04) circle (1) node[shift={(1.1,-.6)}] {$\mathbb{P}$}; 
	\node at (.6,-.4) {$T^+$}; 
	\node at (1.2,-.7) {$T^-$}; 
	\node at (0,-.7) {$C^+$}; 
	\node at (1.4,.2) {$I^-$}; 
	\node at (.6,.3) {$I^+$}; 
	\node at (-.2,.2) {$F^+$}; 
	\node at (.5,-1.4) {$C^-$}; 
	\node at (-1,-2) {$F^-$}; 
\end{tikzpicture}
\end{center}

The values $I^+$ and $I^-$ do represent  very artificial situations. In fact, it is quite reasonable to admit, for example, that mathematical knowledge is always \emph{a priori}\footnote{Kant clearly claims that in \emph{Critique} (B4-B5) and in \emph{Prolegomena} (4:280), but similar thesis can be found in other important philosophers. Descartes doubts of all mathematical certainties througout the fiction of the Evil Genius, but this doubt is, as Gueroult stresses in \cite[p. 20-22]{gue:68}, methodological, metaphysical and hyperbolic. Once it is guaranteed that God could not deceive us, it is impossible to doubt of mathematical truth, unless obvious contradictions (AT 36). Russell seems agreeing with that mathematical knowlege is \emph{a priori}. Indeed, it would be contradictory if some mathematical theorem would be false. According to \cite{dem:cla:05}, Russell's logicism consists in admit that logic and mathematics are both synthetic \emph{a priori} and, as argued in note \footref{russell}, all logical theorems are, for Russell, known \emph{a priori}. }. Assuming this, these modally contradictory scenarios would express that mathematics would be inconsistent. It is hard to believe that most part of philosophers and mathematician would defend it. Anyway, non-deterministic modal semantics could express even those extreme situations.

It is more plausible to admit that our \emph{a priori} knowledge  is always consistent. As it will be shown, it is equivalent to admit that what is necessary is possible. In other words, we can not have propositions that are, at same time, \emph{a priori} true and false . 

That assumption forces a new relationship between the concepts of necessary and possible. Indeed, the concept of necessary will be included in the concept of possible, as expressed by the Venn-diagram below:

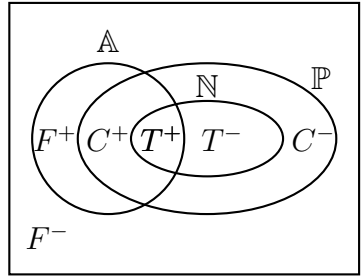
\begin{center}
\begin{tikzpicture}[thick] 
\draw (2.6,-1.8) rectangle (-2,1.8) node[below right] {}; 
\draw (-0.7,0) circle (1) node[above,shift={(0,1)}] {$\mathbb{A}$};
\draw (0.6,0) ellipse (1cm and 0.5cm) node[above,shift={(0,0.4)}] {$\mathbb{N}$};
\draw (0.6,0) ellipse (1.7cm and 1cm) node[shift={(1.5,0.8)}] {$\mathbb{P}$}; 
	\node at (-1.4,0) {$F^+$}; 
	\node at (-0.7,0) {$C^+$}; 
	\node at (0,0) {$T^+$}; 
	\node at (0,0) {$T^+$}; 
	\node at (0.8,0) {$T^-$}; 
	\node at (2,0) {$C^-$};	
	\node at (-1.5,-1.3) {$F^-$}; 
\end{tikzpicture}
\end{center}

In alethic contexts, it seems reasonable to admit that an \emph{a priori} judge can not be falsified by experience. In that sense, what is necessarily true is aways actually true.\footnote{It is easy to check that Kant respects that principle: according to him, anyone that is in an \emph{a priori} domain can not be refuted through experience (B8). Considering that necessity implies \emph{a priori} and considering that, according to our interpretation of modal truth-values, what is actually true is true through experience, that is, \emph{a posteriori}; we can conclude then that it would be contradictory to admit that a sentence would be necessarily true and actually false, since this sentence would be, at the same time, \emph{a priori} and \emph{a posteriori}, what is an absurd.  Kripke explicitly agrees with this principle (see \cite[p. 36]{kri:81}), even though his semantics forces that the relation between the possible worlds is reflexive in order to validate it. Both normal and non-normal modal logics, in Kripkean terms, respects this principle. For more technical results, see \cite{kri:63} and \cite{kri:66}.} In other words, it is hard to accept that something actually true could be impossible.

But alethic interpretation is not the one that makes more sense in a configuration with six truth-values. Since what is obligatory is, in general, permitted, we could use the same diagram below in order to represent the relationship between the concepts of \emph{obligatory}, \emph{permitted} and \emph{actual}. They are represented above by the sets $\mathbb{N}$, $\mathbb{P}$ and $\mathbb{A}$, respectively. In a deontic perspective, the values should be interpreted as follows:

\begin{itemize}
  \item[] \  $T^+$: Obligatory and actually true
  \item[] \  $C^+$: Permitted and actually true
  \item[] \  $F^+$: Forbidden but actually true (infringed) 
  \item[] \  $T^-$: Obligatory but actually false (infringed)
  \item[] \  $C^-$: Permitted and actually false
  \item[] \  $F^-$: Forbidden and actually false 
\end{itemize}

In deontic logics, it makes no sense to consider that what is obligatory is actually the case. Indeed, there are laws but sometimes people do not respect them, and the law is infringed. 

If we come back to an alethic perspective, we could add the principle that what is \emph{a priori} true is, indeed, actually true. Thus, what is necessary would be actual and what is actual would be possible.\footnote{This is essentially the meaning of the modal axiom (T), namely: $\Box p \to p$ and $p \to \Diamond p$.}

Let again $\mathbb{A}$, 
$\mathbb{N}$ and $\mathbb{P}$ represent the concepts of \emph{actual}, \emph{necessary} and \emph{possible}, respectively. The Venn-diagram that represents both inclusion clauses is equivalent to that one:

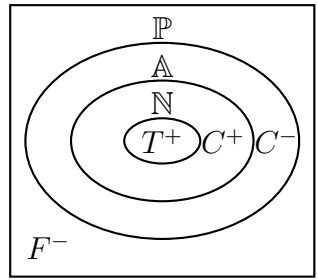
\begin{center}
\begin{tikzpicture}[thick] 
\draw (2,-1.8) rectangle (-2,1.8) node[below right] {}; 
\draw (0,0) ellipse (0.5cm and 0.3cm) node[above,shift={(0,0.7)}] {$\mathbb{A}$};
\draw (0,0) ellipse (1.2cm and 0.8cm) node[above,shift={(0,1.2)}] {$\mathbb{P}$};
\draw (0,0) ellipse (1.8cm and 1.3cm) node[shift={(0,0.5)}] {$\mathbb{N}$}; 
	\node at (0,0) {$T^+$}; 
	\node at (0.8,0) {$C^+$}; 
	\node at (1.5,0) {$C^-$};	
	\node at (-1.5,-1.4) {$F^-$}; 
\end{tikzpicture}
\end{center}

Since necessity entails actuality,  the values $T^-$ and $F^-$ must be eliminated. The four values would be now interpreted as follows:

\begin{itemize}
  \item[] \  $T^+$: necessarily true (or  \emph{a priori} true)
  \item[] \  $C^+$: contingently true (or \emph{a posteriori} true )
  \item[] \  $C^-$: contingently false (or \emph{a posteriori} false )
  \item[] \  $F^-$: necessarily false / impossible (or \emph{a priori} false )
\end{itemize}

But if necessity implies actuality and if actuality implies possibility, we could have, in fact, only two concepts: \emph{actual} and \emph{contingent}. Let us represent them by the sets $\mathbb{A}$ and $\mathbb{C}$, respectively. The Venn-diagram below with only those two concepts corresponds to the last one above:

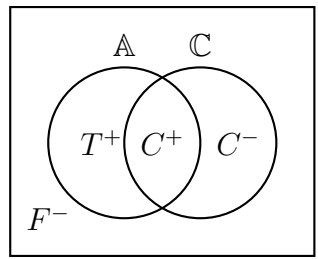
\begin{center}
\begin{tikzpicture}[thick] 
\draw (2.5,-1.5) rectangle (-1.5,1.8) node[below right] {}; 
\draw (0,0) circle (1) node[above,shift={(0,1)}] {$\mathbb{A}$};
\draw (1,0) circle (1) node[above,shift={(0,1)}] {$\mathbb{C}$};
	\node at (-0.3,0) {$T^+$}; 
	\node at (0.5,0) {$C^+$}; 
	\node at (1.5,0) {$C^-$};	
	\node at (-1,-1) {$F^-$}; 
\end{tikzpicture}
\end{center}

We are not claiming here that non-deterministic semantics is the correct way to formulate Kant's  logic. But in the same way that there is an approximation between the Leibnizian  notion of ``the best of possible worlds'' and  relational semantics\footnote{According to Noonan, Kripke semantics ``brought back into the philosophical mainstream the Leibnizian language according to which necessity is truth in all possible
worlds and possibility is truth in some.'', see \cite{noo:13}, p. 8.}, Ivlev's semantics can be seen as an approximation between the notion of necessary and \emph{a priori}, according to a Kantian perspective.
 
\section{Some interpretations of logical operators} \label{operators}

Let us consider an extension of propositional language based on $\neg$ and $\to$, adding $\Box$ and $\Diamond$ as unary connectives.\footnote{We could use $\Box$ as the only modal operator, defining $\Diamond \equiv_{def} \neg \Box \neg$. Reciprocally, we could use $\Diamond$ as the only modal operator, defining $\Box \equiv_{def} \neg \Diamond \neg$. We will consider both as the primitive modal connectives in order to facilitate the readability of modal axioms.} In classical propositional logic, truth-tables are a formal semantics used in order to interpret the linguistic notions of ``deny'' and ``implies''. Modal propositional logic normally extends it, intending to capture the linguistic notion of ``necessary'' and ``possible'' through the symbol $\Box$ and $\Diamond$, respectively.

As usually modal logic are considered as an extension of {\bf PC}, it is expected that formulas that are tautologies in {\bf PC} should continue being  true in any modal semantics. If we are thinking on modal-truth value terms, theorems in {\bf PC} should always receive a designated truth-value. 

At this point, it is convenient to recall the notion of Nmatrix semantics introduced by A.~Avron and I.~Lev. If $\Sigma$ is a propositional signature, then $For(\Sigma)$ will denote the algebra of formulas freely generated by $\Sigma$ from a given denumerable set $\mathcal{V}=\{p_1,p_2,\ldots\}$  of propositional variables. 

\begin{definition} \label{defmulti} Let $\Sigma$ be a propositional signature. 
A {\em multialgebra} (a.k.a. {\em hyperalgebra} or {\em non-deterministic algebra}) over $\Sigma$ is a pair $\mathcal{A}=(A,\sigma_\mathcal{A})$ such that $A$ is a nonempty set (the {\em domain} of $\mathcal{A}$) and $\sigma_\mathcal{A}$  is a function assigning, to each $n$-ary connective $c$ in $\Sigma$, a function (called {\em multioperation} or {\em hyperoperation})  $c^\mathcal{A}:A^n \to \wp(A)-\{\emptyset\}$. 
\end{definition}

\begin{definition} [See~\cite{avr:lev:01}] \label{valNmat} 
Let $\mathcal{M}=(\mathcal{A},D)$ be an Nmatrix over a signature $\Sigma$. A {\em valuation} over $\mathcal{M}$ is a function $v: For(\Sigma)\to A$ such that, for every $n$-ary connective $c$ in $\Sigma$ and every $\varphi_1,\ldots,\varphi_n \in For(\Sigma)$:
$$v(c(\varphi_1,\ldots,\varphi_n)) \in c^\mathcal{A}(v(\varphi_1),\ldots,v(\varphi_n)).$$
\end{definition}

\begin{definition} [See~\cite{avr:lev:01}] \label{semNmat} 
Let $\mathcal{M}=(\mathcal{A},D)$ be an Nmatrix over a signature $\Sigma$, and let $\Gamma \cup \{\varphi\} \subseteq For(\Sigma)$. We say that $\varphi$ is a consequence of $\Gamma$  in the Nmatrix $\mathcal{M}$, denoted by $\Gamma\models_{\mathcal{M}} \varphi$, if the following holds: for every valuation  $v$ over $\mathcal{M}$, if $v[\Gamma] \subseteq D$ then $v(\varphi) \in D$. In particular, $\varphi$ is valid in $\mathcal{M}$, denoted by $\models_{\mathcal{M}} \varphi$, if $v(\varphi) \in D$ for every valuation  $v$ over $\mathcal{M}$.
\end{definition}

From now on the multioperator associated by a multialgebra $\mathcal{A}=(A,\sigma_\mathcal{A})$ to a given connective $c$ will be simply denoted by $c$, when there is no risk of confusion.

Let us consider, for now, four-valued modal scenarios as described in the last section. Let $\Sigma$ be the modal signature containing the connectives $\neg$ (negation), $\Box$ (necessary) and $\to$ (implication), and let $For$ be the set of formulas generated by $\Sigma$ from $\mathcal{V}$.  Consider the following:

\begin{itemize}
  \item[] \  $+ = \{T^+,C^+\}$ (designated values)
  \item[] \  $- = \{C^-, F^-\}$ (non-designated values)
\end{itemize}

In order to extend classical propositional semantics, a four-valued Nmatrix with domain $\{T^+,C^+,C^-, F^-\}$ will be defined. The multioperators associated to the classical connectives $\neg$ and $\to$ of signature $\Sigma$  must respect, at least, the following conditions:
 
\begin{displaymath}
\begin{array}{c c}
\begin{array}{|l|l|}
   	   \hline  	   & \neg \\
\hline \hline T^+  & - \\
 		\hline C^+ & - \\
 		\hline C^- & + \\
 		\hline C^- & + \\
 		\hline	
\end{array}
&

\begin{array}{|l|l|l|l|l|}
 \hline \to  & T^+ & C^+ & C^- 	& F^-\\
   \hline \hline T^+ & + 	  & + 	& - 	& - \\
 		  \hline C^+ & + 	  & + 	& - 	& - \\
 		  \hline C^- & + 	  & + 	& + 	& + \\
 		  \hline C^- & + 	  & + 	& + 	& + \\	
 		  \hline
\end{array}
\end{array}
\end{displaymath}

\

But those restrictions seem  to be very week . Take, for instance, just the multioperator for negation. Consider the sentence:
\begin{center}
(S$_1$) \ \ ``2 plus 2 is equal to 4''
\end{center}
It is clear that (S$_1$) is true. Indeed, it seems not be \emph{a posteriori} true, but  \emph{a priori} true. It is natural to attribute to (S$_1$) the value $T^+$. Consider now the sentence:
\begin{center}
(S$_2$)  \ \ ``it is not the case that 2 plus 2 is not equal to 4''
\end{center}
It is quite reasonable to admit that if (S$_1$) is  \emph{a priori} true, then (S$_2$) would be also  \emph{a priori} true. If we accept, however, the table above, we can have the following situation:

\begin{center}
\begin{tabular}{|c|c|c|}
	\hline
	$p$ 			    	& $\neg p$ 	& $\neg \neg p$ \\
	\hline \hline 		 	&       	& $T^+$ \\	\cline{3-3}
	$T^+$                  & $C^-$		& $C^+$ \\  \cline{2-3} 
	 	   					&  	        & $T^+$ \\	\cline{3-3}
	                   		& $F^-$		& $C^+$ \\
	\hline
\end{tabular}
\end{center}

\

So, despite (S$_1$) being  \emph{a priori} true, (S$_2$) could be  \emph{a posteriori} true (the two rows ending with $C^+$), which would be weird. In order to fix that problem, we must force that the negation of  \emph{a priori} true is  \emph{a priori} false. If we accept this, the same argument can be applied to the other values and we would have the following  deterministic operator for negation:

\begin{displaymath}
\begin{array}{|l|l|}
   	   \hline  	   & \neg \\
\hline \hline  T^+ & \{F^-\} \\
 		\hline C^+ & \{C^-\} \\
 		\hline C^- & \{C^+\} \\
 		\hline F^- & \{T^+\} \\
 		\hline	
\end{array}
\end{displaymath}

\

The argument for constraining the multioperator for implication displayed above is a bit more complex. Take, for instance, two propositions:

\begin{center}
(S$_3$) \ \  $2 = 1 + 1$\\[2mm]

(S$_4$)  \ \ $1+1 = 2$
\end{center}

It is clear that (S$_3$) and (S$_4$) are both  \emph{a priori} true. So, it makes sense to consider that ``(S$_3$) implies (S$_4$)'' will be also  \emph{a priori} true. Consider now the sentence:

\begin{center}
(S$_5$)  \ \ $1 + 1 = 3$
\end{center}

It is clear that (S$_5$) is impossible or, in other words,  it  is  \emph{a priori} false. This forces that (S$_4$) implies  (S$_5$) will be impossible, that is, the implication as a whole must be \emph{a priori} false too.

On the other hand, things are not so simple for contingent situations. Take, for instance, two contingent true sentences:

$$
\begin{array}{ll}
\mbox{(S$_6$)} & \mbox{``Barack Obama is 1.85m height''}\\[2mm]
\mbox{(S$_7$)} & \mbox{``Barack Obama has short hair''}
\end{array}$$

Since both  (S$_6$)  and  (S$_7$) are contingently true, it is expected that ``(S$_6$) implies (S$_7$)'' is only contingently true, that means,  \emph{a posteriori} true. But take account now the sentence:

\begin{center}
  (S$_8$)  \ \ ``Barack Obama is more than 1.80m height''
\end{center}

In this situation, ``(S$_6$) implies (S$_8$)'' is not only contingently true but necessarily either. That is, ``(S$_6$) implies (S$_8$)'' should be  \emph{a priori} true. In order to capture this non-deterministic scenario, we will say that the function ``implies'' attributes to the par of values $\langle C^+, C^+ \rangle$ the set $+$.

Consider now the sentence:

\begin{center}
    (S$_9$)  \ \ ``Barack Obama has not short hair''
\end{center}

It is false, but only contingently. It is clear that both ``(S$_6$) implies (S$_9$)'' and ``(S$_7$) implies (S$_9$)'' will be aways actually false. It makes sense that ``(S$_6$) implies (S$_9$)'' is contingently false. However, someone would say that  ``(S$_7$) implies (S$_9$)'' must be impossible. In fact, how could be possible Obama have and not have short hair? It makes sense, thus, to consider that ``(S$_7$) does not implies (S$_9$)'' is a tautology. So, ``(S$_7$) implies (S$_9$)'' should be impossible if we consider the principle that tautologies are always necessarily true. That is what happens in Kripkean normal modal logics.

Abolishing this rule, we can argue that if ``(S$_7$) implies (S$_9$)''  would be  \emph{a priori} false, then (S$_7$) would be  \emph{a priori} true and that is absurd, because (S$_7$) is contingent. Therefore, the multioperator $\rightarrow$ must always attribute to the ordered pair of values $\langle C^+, C^- \rangle$ the singleton $\{C^-\}$. 

The point here is that a tautology is something that is always true for any valuation instead of ``true in all possible world''. So
``(S$_7$) implies (S$_9$)'' will receive as a set of values always a subset of designated-values, that is, a value in $\{T^+,C^+\}$. But the notion here of \emph{necessary} and \emph{possible} is not a logical one, but an epistemic one. As (S$_7$)
is only empirically true and (S$_9$) is empirically false, ``(S$_7$) implies (S$_9$)'' will be empirically false from an epistemic point of view. In particular, if (S$_9$) is the negation of (S$_7$), thus ``(S$_7$) implies (S$_9$)'' will be always false, because it is a contradiction, that means, it will receive as value always a subset of $-$. But this has nothing to do with the epistemic value of the complex sentence.\footnote{Kripke says: ``The terms `necessarily' and `\emph{a priori}' as applied to statements, are \emph{not} obvious synonyms. There may be a philosophical argument connecting then, perhaps even identifying  them; but an argument is required, not the simple observation that the two terms are clearly interchangeable'' \cite[p. 38]{kri:81}. Analogously, we could sustain that is not obvious that all the tautologies are necessarily true. A philosophical argument is required.}

We can find analogous arguments in order to justify the table below:

\begin{displaymath}
\begin{array}{|l|l|l|l|l|}
 \hline \to  & T^+ 		& C^+ 		& C^- 		& F^-\\
   \hline \hline T^+ & \{T^+\} & \{C^+\} 	& \{C^-\}	& \{F^-\} \\
 		  \hline C^+ & \{T^+\} & + 	  		& \{C^-\} 	& \{C^-\} \\
 		  \hline C^- & \{T^+\} & + 	 	 	& + 		& \{C^-\}\\
 		  \hline F^- & \{T^+\} & \{T^+\} 	& \{T^+ 	& \{T^+\} \\	
 		  \hline
\end{array}
\end{displaymath}

\

Finally, the  multioperator assigned to $\Box$ must capture the notion of \emph{necessary} in natural language, while the multioperator for
$\Diamond$ must capture the notion of \emph{possible}. Thus, consider the sentence:

\begin{center}
    (S$_{10}$)  \ \ ``Barack Obama is mortal''
\end{center}

If (S$_{10}$) is necessarily true, then the sentence

\begin{center}
    (S$_{11}$)  \ \ ``Barack Obama is necessarily mortal''
\end{center}

\noindent
is actually true. But is (S$_{11}$) contingently  or necessarily true? At first, no modal semantics need to compromise with the principle that what is necessary is also necessarily necessary.

Reciprocally, if (S$_{10}$) is not necessarily true, then (S$_{11}$) is actually false. But should be  (S$_{11}$) only contingently false or impossible? At first, no commitment is assured. If we assume that what is necessarily true is impossible to be false, this leads us to the following multioperators:

\begin{displaymath}
\begin{array}{|l|l|l|}
   	   \hline  	   & \Box 	& \Diamond\\
\hline \hline  T^+ & + 		& +\\
 		\hline C^+ & - 		& +\\
 		\hline C^- & - 		& +\\
 		\hline F^- & - 		& - \\
 		\hline	
\end{array}
\end{displaymath}

\

Let $\alpha$ and $\beta$ be formulas of the modal language $For$. The multioperator $\Diamond$ was obtained from $\Box$ and $\neg$ by interpreting, as usual, $\Diamond \alpha = \neg\Box\neg\alpha$. Observe that, by taking the usual notion of composition of multioperators in multialgebras, 

$$\Diamond x = \bigcup\{ \neg z \ : \ z \in  \Box \neg x\}$$

\noindent where $\Box \neg x= \bigcup\{ \Box y \ : \ y \in \neg x\}$ (see for instance~\cite{CFG}).

From the axiomatic point of view, we add to any Hilbert calculus for classical propositional logic  {\bf PC} over the signature $\{\neg, \to\}$ (assuming that {\em Modus Ponens}, \MP,  is the only inference rule)  the following axiom schemas: 

$$
\begin{array}{ll}
\axK & \Box(\alpha\to \beta) \to (\Box \alpha \to \Box \beta)\\[2mm]
\axKu & \Box(\alpha \to \beta) \to (\Diamond \alpha \to \Diamond \beta)\\[2mm]
\axKd & \Diamond(\alpha \to \beta) \to (\Box \alpha \to \Diamond \beta)\\[2mm]
\axMu & \neg \Diamond \alpha \to \Box(\alpha \to \beta)\\[2mm]
\axMd & \Box \beta  \to \Box(\alpha \to \beta)\\[2mm]
\axMt & \Diamond \beta \to \Diamond (\alpha \to \beta)\\[2mm]
\axMc & \Diamond \neg \alpha \to \Diamond(\alpha \to \beta)\\[2mm] 
\axT & \Box \alpha \to \alpha\\[2mm]
\axDNu & \Box \alpha \to \Box \neg\neg \alpha\\[2mm]
\axDNd & \Box \neg\neg \alpha \to \Box \alpha\\[2mm]
\end{array}$$

The Hilbert calculus composed by these axiom schemas (and having \MP\ as the only inference rule) is sound and complete with respect to the four-valued Nmatrix presented above, and it is called {\bf Tm} (see~\cite{con:cer:per:15,con:cer:per:17}).

Let now assume the principle that what is necessary is also necessarily necessary. In order to do this, and keeping the multioperators for negation and  implication, we slightly modify the multioperator for $\Box$ (and hence the one for $\Diamond$) as follows:

\begin{displaymath}
\begin{array}{|l|c|c|}
   	   \hline  	   & \Box 			& \Diamond\\
\hline \hline  T^+ & \{T^+\}		& +\\
 		\hline C^+ & - 				& +\\
 		\hline C^- & - 				& +\\
 		\hline F^- & - 				& \{F^-\} \\
 		\hline	
\end{array}
\end{displaymath}

\

From the axiomatic point of view, we add the axiom schema:
$$
\begin{array}{ll}
\axfour & \Box \alpha \to \Box \Box \alpha
\end{array}$$

The system obtained is {\bf T4m}, and it is sound and complete w.r.t. the corresponding four-valued Nmatrix (see~\cite{con:cer:per:15,con:cer:per:17}).

We could assume, reciprocally, that what is not necessary is impossible to be necessary. That forces us to assume  deterministic operators for the modal connectives:

\begin{displaymath}
\begin{array}{|l|c|c|}
   	   \hline  	   & \Box 			& \Diamond\\
\hline \hline  T^+ & \{T^+\}		& \{T^+\}\\
 		\hline C^+ & \{F^-\} 		& \{T^+\}\\
 		\hline C^- & \{F^-\} 		& \{T^+\}\\
 		\hline F^- & \{F^-\} 		& \{F^-\} \\
 		\hline	
\end{array}
\end{displaymath}

Consider now the system obtained from {\bf T4m} by adding the following axiom schema:
$$
\begin{array}{ll}
\axfive & \Diamond \Box \alpha \to \Box \alpha
\end{array}$$

The resulting system, called {\bf T45m}, is  sound and complete w.r.t. the corresponding four-valued Nmatrix (see~\cite{con:cer:per:15,con:cer:per:17}).\footnote{The systems {\bf Tm} and {\bf T45m} are originally called $Sa^+$ and $Sb^+$, respectively. See \cite{ivl:88}.}

\section{Interpreting modal operators in a deontic context} \label{deon}

We might ask what would happen with those Nmatrices if we abandon the principle that what is  \emph{a priori} true  is actually true. From an axiomatic point of view, we can weaken \axT, replacing it by the following axiom schema:
$$
\begin{array}{ll}
\axD & \Box \alpha \to \Diamond \alpha
\end{array}$$

This system was called {\bf Dm} in~\cite{con:cer:per:15,con:cer:per:17}. From a semantic point of view, we will have six truth-values, as discussed in Section~\ref{values}.

In a deontic scenario, it is quite natural to expect the same behavior that tautologies had in a alethic scenario. That is, all propositional tautologies will always receive a designated value. In this setting, the designed values are:

\begin{itemize}
  \item[] \  $+ = \{T^+,C^+, F^+\}$ (designated values)
  \item[] \  $- = \{T^-,C^-, F^-\}$ (non-designated values)
\end{itemize}

By reasoning as in the last section,  the following deterministic operator for  negation should be accepted:

\begin{displaymath}
\begin{array}{|l|l|}
   	   \hline  	   & \neg \\
\hline \hline  T^+ & \{T^-\} \\
 		\hline C^+ & \{C^-\} \\
 		\hline F^+ & \{F^-\} \\
 		\hline T^- & \{F^+\} \\
 		\hline C^- & \{C^+\} \\
 		\hline F^- & \{T^+\} \\
 		\hline	
\end{array}
\end{displaymath}

Clearly, if the sentence

\begin{center}
(S$_{12}$) Drive safely
\end{center}
is obliged, thus ``Do not drive safely'' will be forbidden and so on. The multioperator for implication is not quite obvious, as it can be seen below:

\begin{displaymath}
\begin{array}{|c|c|c|c|c|c|c|}
      \hline \to 	& T^+ 	  & C^+ 			& F^+ 		& T^- 		& C^- 			& F^-\\
\hline \hline     		T^+ & \{T^+\} & \{C^+\} 		& \{F^+\} 	& \{T^-\}	& \{C^-\} 		& \{F^-\}\\
       \hline     		C^+ & \{T^+\} & \{T^+,C^+\} 	& \{C^+\} 	& \{T^-\} 	& \{T^-,C^-\} 	& \{C^-\}\\
       \hline     		F^+ & \{T^+\} & \{T^+\} 		& \{T^+\} 	& \{T^-\} 	& \{T^-\} 		& \{T^-\}\\
      \hline     		T^- & \{T^+\} & \{C^+\} 		& \{F^+\} 	& \{T^+\} 	& \{C^+\} 		& \{F^+\}\\
       \hline     		C^- & \{T^+\} & \{T^+,C^+\}		& \{C^+\} 	& \{T^+\} 	& \{T^+,C^+\} 	& \{C^+\}\\
       \hline     		F^- & \{T^+\} & \{T^+\} 		& \{T^+\} 	& \{T^+\} 	& \{T^+\} 		& \{T^+\}\\
       \hline
\end{array}
\end{displaymath}

\

However, we have three arguments in order to defend the relevance of that multioperator. First, from a semantic point of view, by considering the Nmatrix formed by the  negation operator and the multioperator for $\to$ introduced above, together with some quite natural multioperators for $\Box$ and $\Diamond$ (to be introduced below), were $+$ is the set of designated values, we obtain a sound and  complete Nmatrix deontic semantic for {\bf Dm}. Secondly, from an axiomatic point of view, the difference between {\bf Dm} and {\bf Tm} is analogous to the one between the Kripkean systems {\bf T} and {\bf D}, in which the relation between the possible worlds is reflexive and serial, respectively.\footnote{The system {\bf T} is characterized by the axiom (T), the necessitation rule and the axiom (K). The completenes result for {\bf T} was first proved by Kripke in \cite{kri:63}. Axiomatics for {\bf D} was proposed by Lemmon in \cite{lem:66}, where he called it {\bf T(D)}. We follow Segerberg's nomenclature in \cite[p. 47]{seg:71}. Kripke did not originally proove {\bf D}-completeness with respect to his semantics, this result is for the first time proved in \cite[p. 47-50]{seg:71} .\label{D}} Third, eliminating the new values $T^-$ and $F^+$, we have exactly the Nmatrix for {\bf Tm}. As we shall see in Section~\ref{DAT}, this implies that the multialgebra underlying the Nmatrix  for {\bf Tm} is a submultialgebra of that for {\bf Dm}. 

Somebody could ask the linguistic justification in order to add a new case of non-determinism: when an antecedent is permitted and true and the consequence is permitted and false. For that, consider the sentences

$$
\begin{array}{ll}
\mbox{(S$_{13}$)} & \mbox{John takes Peter's sheep.}\\[2mm]
\mbox{(S$_{14}$)} &  \mbox{John smokes a cigarette.}\\[2mm]
\mbox{(S$_{15}$)} & \mbox{John pays for Peter's sheep.}
\end{array}
$$
All those sentences are clearly permitted. Suppose that John smokes a cigarette, takes Peter's sheep but doesn't pay him for it. It is clear that 
``(S$_{14}$) implies (S$_{15}$)'' is permitted while ``(S$_{13}$) implies (S$_{15}$)'' is an infraction.  Therefore, it seems reasonable that the multioperator $\rightarrow$ attributes in a deontic context to the ordered pair of values $\langle C^+, C^- \rangle$ the set of value $\{T^-,C^-\}$. 

Two extension of {\bf Dm} were considered in~\cite{con:cer:per:15,con:cer:per:17}. If we force that obligations are always obligatory,  we add the schema axiom \axfour\ to {\bf Dm}, obtaining the system {\bf D4m}. Reciprocally, if we accept too that what is not obligatory is not permitted being obligatory, we add the schema axiom \axfive\  to {\bf D4m}, obtaining {\bf D45m}.

Analogously to the alethic context, the semantic difference between {\bf Dm}, {\bf D4m} and {\bf D45m} is only with respect to modal operators $\Box$ and $\Diamond$:

\begin{displaymath}
\begin{array}{|l|c|c|c|c|c|c|}
		\hline     & \multicolumn{2}{|l|} {\bf Dm} & \multicolumn{2}{|l|}  {\bf D4m} 	&  \multicolumn{2}{|l|}   {\bf D45m} 		\\
   	   \hline  	   & \Box 	& \Diamond	&  \Box 	& \Diamond  & \Box 		& \Diamond \\
\hline \hline  T^+ & + 		& +			&\{T^+\}	& + 	& \{T^+\}	& \{T^+\}\\
 		\hline C^+ & - 		& + 		& - 		& + 	& \{F^-\} 	& \{T^+\}\\
		\hline F^+ & - 		& -			& - 		& \{F^-\} 	& \{F^-\}	& \{F^-\} \\		
		\hline T^- & +		& + 		& \{T^+\}		& +	& \{F^-\}	& \{T^+\}\\
 		\hline C^- & - 		& +			& - 		& +			& \{F^-\} 	& \{T^+\}\\
 		\hline F^- & - 		& - 		& - 		& \{F^-\}  		& \{F^-\} 	& \{F^-\} \\
 		\hline	
\end{array}
\end{displaymath}

\

The resulting Nmatrices are sound and complete w.r.t. the corresponding Hilbert calculi (see~\cite{con:cer:per:15,con:cer:per:17}). 

\section{Some 8-valued modal systems} \label{K-compl}

The next natural step would be trying to extend the proposed analogy between  Kripke semantics and modal Nmatrices. In Kripke semantics, in order to characterize the system {\bf D}, the relation between the possible worlds must be serial; by eliminating the axiom \axD\ it is obtained the basic system {\bf K}, in which the relation between the possible worlds is empty. By recalling our analysis in Section~\ref{values} with respect to eight truth-values, and by considering the possibility of  adding to {\bf K} the axioms \axfour\ and \axfive, Ivlev-like non-normal modal systems can be  obtained, namely {\bf Km}, {\bf K4m} and {\bf K45m}, respectively. It could be natural to consider, in our framework, the following multioperators for negation and the modalities:\footnote{In~\cite[Section~5]{omo:sku:16} it was proposed another Ivlev-like non-normal modal sytem characterized by a different 8-valued modal Nmatrix, called {\bf K}. See more comments below.}

\begin{displaymath}
\begin{array}{|l|c|c|c|c|c|c|c|}
		\hline \multicolumn{2}{|l|}  { } 			&\multicolumn{2}{|l|} {\bf Km} & \multicolumn{2}{|l|}  {\bf K4m} 	&  \multicolumn{2}{|l|}   {\bf K45m} 		\\
   	   \hline  	   & \neg 		&\Box 	& \Diamond	&  \Box 		& \Diamond  	& \Box 			& \Diamond \\
\hline \hline  T^+ & \{F^-\} 	& +		& + 		& \{T^+, I^+\}	& +				& \{T^+, I^+\} 	& \{T^+, I^+\}	\\
 		\hline C^+ & \{C^-\} 	& -		& + 		& - 			& + 			& \{F^-, I^-\} 	& \{T^+, I^+\} \\
		\hline F^+ & \{T^-\} 	& -		& - 		& - 			& \{F^-, I^-\}	& \{F^-, I^-\} 	& \{F^-, I^-\}\\		
		\hline I^+ & \{I^-\} 	& +     & -			& \{T^+, I^+\}	& \{F^-, I^-\}				& \{T^+ I^+\}	& \{F^-, I^-\}\\		
		\hline T^- & \{F^+\}	& + 	& + 		& \{T^+, I^+\}	& +				& \{T^+, I^+\}	& \{T^+, I^+\}\\
 		\hline C^- & \{C^+\} 	& -		& + 		& -				& + 			& \{F^-, I^-\}	& \{T^+, I^+\}\\
 		\hline F^- & \{T^+\} 	& - 	& - 		& - 			& \{F^-, I^-\}	& \{F^-, I^-\}	& \{F^-, I^-\}\\
		\hline I^- & \{I^+\} 	& +     & -			& \{T^+, I^+\} 	& \{F^-, I^-\}	& \{T^+, I^+\} 	& \{F^-, I^-\}\\		
	 	\hline	
\end{array}
\end{displaymath}

\

From the axiomatic point of view, it is also expected that, by eliminating the axiom \axD\ in {\bf Km}, we can provide a sound and complete system coherent w.r.t. the multioperators for 
$\Box$ and $\Diamond$ above. This is not the case, at it will be shown below:

\begin{exam} {\em
Let $v$ be a valuation function  over the proposed Nmatrix for {\bf Km}, that is, in which the multioperators for $\neg$ and $\Box$ are as displayed above.
 Let {\bf S} be any modal system that extends {\bf PC} and derives the axioms \axMu\ and \axK. Suppose that $v(p) = \{I^+\}$ and $v(q) = \{C^+\}$ for two given propositional variables $p$ and $q$. Thus, by the $\neg$-multioperator, $v(\neg p) = \{I^-\}$. In addition, by the $\Box$-multioperator for {\bf Km}, $v(\Box \neg p) \in +$ and, by \axMu\ and the fact that {\bf S} extends {\bf PC}, we would have that $v(\Box(p \to q)) \in +$. On the other hand, by the $\Box$-multioperator for {\bf Km} once again, we  would have that $v(\Box p) \in +$ and 
$v(\Box q) \in -$. Since {\bf S} extends {\bf PC}, we would have that $v(\Box(p\to q) \to (\Box p \to \Box q)) \in -$, falsifying axiom \axK.}
\end{exam}

\

At this point, there are two possible ways to go: a syntactical (axiomatic) one and a semantical one. The axiomatic one is to abandon either \axK\ or \axMu\ or both in order to obtain completeness.\footnote{This was the way chosen by \cite{omo:sku:16} when defining their above mentioned 8-valued modal logic. The proposed system did not only abolish  axioms \axK\ and \axMu, but also axiom \axKu. It is worth noting that not only \axK\ but also \axKu\ and \axMu\ are valid  in any Kripke model; in particular, they hold in all models whose accessibility relation is empty.}  If the key idea is to propose a weaker version of Kripke's system {\bf K} without necessitation, it seems very weird that the only modal axiom of this system does not hold. Besides, both \axK\ and \axMu\ are quite natural: the first one says that if an implication is necessary and the antecedent is also necessary, the consequent will be necessary too; the second one says that if the antecedent of an implication is impossible, the implication  will be necessary. 

These are the reasons to withdraw from the axiomatic way (at least from the perspective of dropping axioms from {\bf K}), and trying the semantic one. In a relational model in which the accessibility relation is empty, all atomic formulas are necessary. By \axK\ we conclude that all implications will also be necessary, that is, for any formula $\alpha$ and $\beta$, the formulas $\Box(\alpha \to \beta)$ and  $\Box\neg (\alpha \to \beta)$ are both true in any world in those models. Analogous to Kripkean semantics, in our perspective the values $I^+$ and $I^-$ must ``infect'' all the propositional formulas, that means, they propagate by means of implication as can be verified by the Nmatrix below:

\begin{displaymath}
\begin{array}{|c|c|c|c|c|c|c|c|c|}
      \hline \to 	& T^+ 	& C^+ 		 	 & F^+ 	   & I^+ 	  & T^- 	& C^- 				& F^- 	  & I^-\\
\hline \hline     		T^+ & \{T^+\} & \{C^+\} 		 & \{F^+\} & \{I^+\} & \{T^-\} & \{C^-\} 			& \{F^-\} & \{I^-\}\\
       \hline     		C^+ & \{T^+\} & \{T^+,C^+\} 	 & \{C^+\} & \{I^+\} & \{T^-\} & \{T^-,C^-\} 		& \{C^-\} & \{I^-\}\\
       \hline     		F^+ & \{T^+\} & \{T^+\} 		 & \{T^+\} & \{I^+\} & \{T^-\} & \{T^-\} 			& \{T^-\} & \{I^-\}\\
		\hline			I^+ & \{I^+\} & \{I^+\} 		 & \{I^+\} & \{I^+\} & \{I^-\} & \{I^-\} 			& \{I^-\} & \{I^-\}\\    
       \hline     		T^- & \{T^+\} & \{C^+\} 		 & \{F^+\} & \{I^+\} & \{T^+\} & \{C^+\} 			& \{F^+\} & \{I^+\}\\
       \hline     		C^- & \{T^+\} & \{T^+,C^+\} 	 & \{C^+\} & \{I^+\} & \{T^+\} & \{T^+,C^+\} 		& \{C^+\} & \{I^+\}\\
       \hline     		F^- & \{T^+\} & \{T^+\} 		 & \{T^+\} & \{I^+\} & \{T^+\} & \{T^+\} 			& \{T^+\} & \{I^+\}\\
       \hline			I^- & \{I^+\} & \{I^+\} 		 & \{I^+\} & \{I^+\} & \{I^+\} & \{I^+\} 			& \{I^+\} & \{I^+\}\\    
       \hline
\end{array}
\end{displaymath}

\

Note that if we eliminate the values $I^+$ and $I^-$ we obtain exactly the Nmatrices for {\bf Dm}. Our task now is to weaken {\bf Dm}-axiomatic in order to obtain completeness. 

Given that {\bf Km} extends {\bf PC},  disjunction and conjunction can be defined as follows: $\alpha \vee \beta := \neg \alpha \to \beta$ and $\alpha \wedge \beta := \neg(\alpha \to \neg\beta)$. As a consequence of this, the corresponding multioperators are defined as follows:

\begin{displaymath}
\begin{array}{|c|c|c|c|c|c|c|c|c|}
      \hline \vee 	& T^+ 	& C^+ 		 	 & F^+ 	   & I^+ 	  & T^- 	& C^- 				& F^- 	  & I^-\\
\hline \hline
T^+ & \{T^+\} & \{T^+\} 		 & \{T^+\} & \{I^+\} & \{T^+\} & \{T^+\} 			& \{T^+\} & \{I^+\}\\
       \hline
C^+ & \{T^+\} & \{T^+,C^+\} 	 & \{C^+\} & \{I^+\} & \{T^+\} & \{T^+,C^+\} 		& \{C^+\} & \{I^+\}\\
       \hline
F^+ & \{T^+\} & \{C^+\} 		 & \{F^+\} & \{I^+\} & \{T^+\} & \{C^+\} 			& \{F^+\} & \{I^+\}\\
		\hline
I^+ & \{I^+\} & \{I^+\} 		 & \{I^+\} & \{I^+\} & \{I^+\} & \{I^+\} 			& \{I^+\} & \{I^+\}\\    
       \hline
T^- & \{T^+\} & \{T^+\} 		 & \{T^+\} & \{I^+\} & \{T^-\} & \{T^-\} 			& \{T^-\} & \{I^-\}\\
       \hline
C^- & \{T^+\} & \{T^+,C^+\} 	 & \{C^+\} & \{I^+\} & \{T^-\} & \{T^-,C^-\} 		& \{C^-\} & \{I^-\}\\
       \hline
F^- & \{T^+\} & \{C^+\} 		 & \{F^+\} & \{I^+\} & \{T^-\} & \{C^-\} 			& \{F^-\} & \{I^-\}\\
       \hline
I^- & \{I^+\} & \{I^+\} 		 & \{I^+\} & \{I^+\} & \{I^-\} & \{I^-\} 			& \{I^-\} & \{I^-\}\\    
       \hline
\end{array}
\end{displaymath}

\

\begin{displaymath}
\begin{array}{|c|c|c|c|c|c|c|c|c|}
      \hline \wedge 	& T^+ 	& C^+ 		 	 & F^+ 	   & I^+ 	  & T^- 	& C^- 				& F^- 	  & I^-\\
\hline \hline
T^+ & \{T^+\} & \{C^+\} 		 & \{F^+\} & \{I^+\} & \{T^-\} & \{C^-\} 			& \{F^-\} & \{I^-\}\\
       \hline
C^+ & \{C^+\} & \{F^+,C^+\} 	 & \{F^+\} & \{I^+\} & \{C^-\} & \{F^-,C^-\} 		& \{F^-\} & \{I^-\}\\
       \hline
F^+ & \{F^+\} & \{F^+\} 		 & \{F^+\} & \{I^+\} & \{F^-\} & \{F^-\} 			& \{F^-\} & \{I^-\}\\
		\hline
I^+ & \{I^+\} & \{I^+\} 		 & \{I^+\} & \{I^+\} & \{I^-\} & \{I^-\} 			& \{I^-\} & \{I^-\}\\    
       \hline
T^- & \{T^-\} & \{C^-\} 		 & \{F^-\} & \{I^-\} & \{T^-\} & \{C^-\} 			& \{F^-\} & \{I^-\}\\
       \hline
C^- & \{C^-\} & \{F^-,C^-\} 	 & \{F^-\} & \{I^-\} & \{C^-\} & \{F^-,C^-\} 		& \{C^-\} & \{I^-\}\\
       \hline
F^- & \{F^-\} & \{F^-\} 		 & \{F^-\} & \{I^-\} & \{F^-\} & \{F^-\} 			& \{F^-\} & \{I^-\}\\
       \hline
I^- & \{I^-\} & \{I^-\} 		 & \{I^-\} & \{I^-\} & \{I^-\} & \{I^-\} 			& \{I^-\} & \{I^-\}\\    
       \hline
\end{array}
\end{displaymath}

\

\

\begin{notation} \label{nota}
{\em
From now on,  the following notation for axiom  schemas will be adopted: if \textsf{(A)} is the name of an axiom schema, and $\alpha_1,\ldots,\alpha_n$ is the list of formula schemas involved in it (reading from left to right), then we will write $A(\alpha_1,\ldots,\alpha_n)$ to refer to the formula schema labelled by that axiom. By abuse of notation, if $\alpha_1,\ldots,\alpha_n$ are concrete formulas then we will also write $A(\alpha_1,\ldots,\alpha_n)$ to denote a concrete instance of the formula schema  $A(\alpha_1,\ldots,\alpha_n)$ (when there is no risk of confusion). For instance, given formula schemas $\alpha$ and $\beta$ then $K(\alpha,\beta)$ will denote the formula schema $\Box(\alpha \to \beta) \to (\Box \alpha \to \Box \beta)$ of \axK, while $M3(\beta,\alpha)$ will denote the formula schema $\Diamond \beta \to \Diamond (\alpha \to \beta)$ labelled by \axMt. Both formula schemas will denote concrete formulas whenever $\alpha$ and $\beta$ denote concrete formulas. }
\end{notation}

\

Now, consider the following axiom schemas:

$$
\begin{array}{ll}
\axKp & \Diamond  \alpha \to K(\alpha,\beta)\\[2mm]
\axKup & \Diamond \neg \beta \to K1(\alpha,\beta)\\[2mm]
\axKdp & \Diamond \alpha \to  K2(\alpha,\beta)\\[2mm]
\axMtp & (\Diamond \alpha \vee \Diamond \neg \alpha) \to M3(\beta,\alpha)\\[2mm]
\axMcp & \Diamond \neg \beta \to  M4(\alpha,\beta)\\[2mm]
\axIu & (\Box \alpha \wedge \Box \neg \alpha) \to (\Box(\alpha \to \beta) \wedge \Box 
					\neg(\alpha \to \beta))\\[2mm]
\axId & (\Box \beta \wedge \Box \neg \beta) \to (\Box(\alpha \to \beta) \wedge \Box 
					\neg(\alpha \to \beta))
\end{array}$$

The system {\bf Km} is obtained by adding to {\bf PC} the seven axiom schemas above, together with axiom schemas \axMu, \axMd, \axDNu\ and \axDNd.

As usual, if $\Gamma$ is a set of formulas and $\alpha$ is a formula,  $\Gamma \vdash_{\bf Km} \alpha$ means that there is a {\bf Km}-proof of $\alpha$ from formulas in $\Gamma$. Besides,  $\Gamma \vDash_{\bf Km} \alpha$ means that every {\bf Km}-valuation that simultaneously assigns to every formula in $\Gamma$ an element of $+$, will also assigns to $\alpha$ an element of $+$. This is the usual definition of consequence relation in Nmatrices (recall Definition~\ref{semNmat}).

\begin{theorem}[Soundness of {\bf Km}] \label{SoundKm}
Let $\Gamma$ be a set of formulas and let $\alpha$ be a formula. Then
$$\Gamma \vdash_{\bf Km} \alpha \textrm{ implies } \Gamma \vDash_{\bf Km} \alpha$$
\end{theorem}

\begin{proof} We begin by proving that any instance of an axiom schema of {\bf Km} is valid in its Nmatrix. For propositional tautologies, just check that the Nmatrix for {\bf Km} respects the classical truth-tables of  the operators $\neg$ and $\to$, by considering $+$ as {\em true} and $-$ as {\em false}. Namely: 

\begin{displaymath}
\begin{array}{|c|c|}
      \hline  	& \neg\\
\hline \hline     		
+ & -\\
       \hline     		
- & +\\    
       \hline
\end{array}
\hspace{2cm}
\begin{array}{|c|c|c|}
      \hline \to 	& + & -\\
\hline \hline     		
+ & + & -\\
       \hline     		
- & + & +\\    
       \hline
\end{array}
\end{displaymath}

\

\noindent 
For the modal axioms:

\begin{description}
 	\item[(K')] Suppose, by absurd, that $v(K'(\alpha,\beta))\in -$ for some formulas $\alpha$ and $\beta$. Then $v(\Diamond \alpha) \in +$, $v(\Box \alpha) \in +$,  
 	$v(\Box(\alpha \to \beta)) \in +$ but $v(\Box\beta)\in-$. First, we have that
 	$v(\alpha) \in (T \cup C) \cap (T \cup I)$, hence $v(\alpha) \in T$. Besides, we have that $v(\alpha \to \beta) \in T \cup I$ since $v(\Box(\alpha \to \beta)) \in +$. But  $v(\Box\beta) \in -$ implies that  $v(\beta) \in C \cup F$ and so, by definition of the multioperator $\to$ in {\bf Km},  $v(\alpha \to \beta)  \in v(\alpha) \to v(\beta) \subseteq  C \cup F$, a contradiction.
 	\item[(K1')] Suppose now that $v(K1'(\beta,\alpha))\in -$ for some formulas $\alpha$ and $\beta$. Hence, $v(\Diamond \neg \beta) \in +$, 
 	$v(\Box(\alpha \to \beta)) \in +$, $v(\Diamond \alpha) \in +$ but $v(\Diamond \beta) \in -$. From $v(\Diamond \neg \beta) \in +$ it follows that $v(\neg\beta) \in T \cup C$, hence  $v(\beta) \in F \cup C$. But $v(\Diamond \beta) \in -$, then $v(\beta) \in F\cup I$. Thus, $v(\beta) \in F$. In addition, $v(\alpha \to \beta) \in T \cup I$ and $v(\alpha) \in T \cup C$. But, by the definition of $\to$, we have that $v(\alpha \to \beta) \in v(\alpha) \to v(\beta) \subseteq C \cup F$, a contradiction.
 	\item[(K2')] Consider, by absurd, that $v(\Diamond \alpha) \in +$, $v(\Diamond (\alpha \to \beta)) \in +$,
 	$v(\Box\alpha) \in +$ but $v( \Diamond \beta) \in -$. As it was proven in the case for \axK, it follows that $v(\alpha) \in T$. Since $v(\Diamond(\alpha \to \beta))\in +$, we have 
that $v(\alpha \to \beta) \in T \cup C$. From $v(\Diamond \beta) \in -$ it follows that $v(\beta) \in F \cup I$. Then, by the definition of $\to$, $v(\alpha \to \beta) \in F \cup I$, a contradiction.
 	\item[(M3')] Suppose that $v(M3'(\alpha,\beta))\in -$ for some formulas $\alpha$ and $\beta$. That is, $v(\Diamond \alpha \vee \Diamond \neg \alpha) \in +$ but $v(M3(\beta,\alpha)) \in -$. From the latter, we infer that $v(\Diamond \beta) \in +$ and so $v(\beta) \in T \cup C$; however, we also conclude that $v(\Diamond (\alpha \to \beta)) \in -$, hence
 	$v(\alpha \to \beta) \in F \cup I$. From this and the definition of $\to$ we get that $v(\alpha) \in I$.  But this implies that both $v(\Diamond \alpha) \in -$ and $v(\Diamond \neg \alpha) \in -$, and thus $v(\Diamond \alpha \vee \Diamond \neg \alpha) \in -$, a contradicion.
 	\item[(M4')] Suppose now that $v(M4'(\beta,\alpha))\in -$ for some formulas $\alpha$ and $\beta$. Thus, $v(\Diamond \neg \beta) \in +$, $v(\Diamond \neg \alpha) \in +$ but 
 	$v(\Diamond ( \alpha \to \beta)) \in -$. As it was proven in the previous cases, it follows that  $v(\beta) \in T \cup C$ and $v(\alpha) \in T \cup C$, but 
 	$v(\alpha \to \beta) \subseteq F \cup I$.  However, using the definition of $\to$ it also follows that $v(\alpha \to \beta)  \in v(\alpha) \to v(\beta) \subseteq  T \cup C$, a contradiction.
 	\item[(I1)] Just note that if $v(\Box \alpha \wedge \Box \neg \alpha) \in +$ then both $v(\Box \alpha) \in +$ and $v(\Box \neg \alpha) \in +$. Hence $v(\alpha) \in I$. From this, $v(\alpha \to \beta) \in I$, therefore $v(\Box (\alpha \to \beta) \wedge \Box \neg (\alpha \to \beta)) \in +$.
 	\item[(I2)] If $v(\Box \beta \wedge \Box \neg \beta) \in +$ then, reasoning as in the case of~\axIu, $v(\beta) \in I$. The rest of the proof is also analogous to the last case.
 	\item[(DN1)-(DN2)] $v(\Box \alpha) \in +$ iff $v(\alpha) \in T\cup I$ iff $v(\neg \neg \alpha) \in T \cup I$ iff $v(\Box \neg \neg \alpha) \in +$. 
 \end{description}

In addition, by observing the truth-table for $\to$ over $\{+,-\}$ displayed above, it follows that \MP\ preserves trueness, namely: if $v(\alpha) \in +$ and $v(\alpha \to \beta) \in +$ then $v(\beta) \in +$, for every valuation $v$ over the Nmatrix of {\bf Km}.

Now, assume that $\Gamma \vdash_{\bf Km} \alpha$ and let $v$ be a valuation over the Nmatrix of {\bf Km} such that $v(\gamma) \in +$ for every $\gamma \in \Gamma$. Taking into account that any axiom of {\bf Km} is valid and that \MP\ preserves trueness,  it can be proven by induction on the length of a derivation in {\bf Km} of   $\alpha$ from $\Gamma$ that $v(\alpha) \in +$. Therefore $\Gamma \vDash_{\bf Km}\alpha$. 
\end{proof}

\
\begin{theorem}[Deduction metatheorem (DMT)] \

	Let $\Gamma$ be a set of formulas and let both $\alpha$ and $\beta$ be
	formulas. Thus:
	
	$$\Gamma, \alpha \vdash_{\bf Km} \beta \textrm{ iff } \Gamma \vdash_{\bf Km} \alpha \to \beta$$
\end{theorem}

\begin{proof}
Analogous to \cite[p. 28]{men:10}, since {\bf Km} is an axiomatic extension of {\bf PC}, hence \MP\ is the only inference rule.
\end{proof}

\begin{definition}
Consider the following relation $\equiv$  defined on $For$: $\alpha \equiv \beta$ if and only if both $\alpha \to \beta$ and $\beta \to \alpha$ are derivable in {\bf Km}.
\end{definition}

It is easy to see that $\equiv$ is an equivalence relation. It is not a congruence since the connective $\Box$ does not preserve logical equivalences, as observed by Omori and Skurt in~\cite[Observation~72]{omo:sku:16} for the stronger system {\bf Tm} (called {\bf T} by the authors).

\begin{lemma} \label{A}  
Let $\alpha$ be a formula in $For$. Then, the following holds in {\bf Km} (recallling that $\Diamond \alpha$ is a notation for the formula $\neg\Box\neg\alpha$):
\begin{enumerate}[(i)]
\item $\Diamond\neg\alpha \equiv \neg \Box \alpha$;
\item $\Box\alpha \equiv \neg\Diamond\neg  \alpha$;
\item $\Box\neg\alpha \equiv \neg \Diamond \alpha$;
\item $\Diamond\alpha \equiv \Diamond\neg\neg \alpha$.
\end{enumerate}
\end{lemma}
\begin{proof}
Analogous to~\cite[Lemma~B]{con:cer:per:17}. 
\end{proof}

\begin{lemma} \label{B}
Let $\Gamma \cup \{\alpha,\beta\}$ be formulas of formulas in $For$. Then, the following holds: 
\begin{enumerate}[(i)]
	\item $\Box \alpha, \Diamond \alpha, \Diamond \neg \beta \vdash_{\bf Km} \Diamond \neg (\alpha \to \beta)$
	\item $\Diamond \alpha, \Box \neg \beta, \Diamond \neg \beta \vdash \Diamond \neg (\alpha \to \beta)$
	\item $\Box \alpha, \Diamond \alpha, \Box \neg \beta \vdash_{\bf Km} \Box \neg(\alpha \to \beta)$
	\item $\Diamond \alpha, \Diamond \beta \vdash_{\bf Km} \Diamond(\alpha \to \beta)$
	\item  $\Diamond \neg \alpha, \Diamond \beta \vdash_{\bf Km} \Diamond(\alpha \to \beta)$
	\item $\Diamond \neg \alpha, \Diamond \neg \beta \vdash_{\bf Km} \Diamond(\alpha \to \beta)$
\end{enumerate}
\end{lemma}
\begin{proof} \ \\[1mm]
(i)  By \axKp, {\bf PC}, (DMT) and Lemma~\ref{A}~(i).\\[1mm]
(ii)  By \axKup, {\bf PC}, (DMT) and Lemma~\ref{A}~(i) and~(iii).\\[1mm]
(iii) By \axKdp, {\bf PC}, (DMT) and Lemma~\ref{A}~(iii).\\[1mm]
(iv) By \axMtp, {\bf PC} and (DMT).\\[1mm]
(v) Analogous to (iv).\\[1mm]
(vi) By \axMcp, {\bf PC} and (DMT).
\end{proof}

\

In order to prove completeness, some standard definitions and results form propositional logic will be recalled now  (see, for instance, \cite{woj:84}). 

A logic $\mathcal{L}$ defined over a $\mathbb{L}$ with a
consequence relation $\vdash_{\mathcal{L}}$ is said to be  {\em Tarskian} if it satisfies
the following, for every $\Gamma \cup \Delta \cup
\{\alpha\}\subseteq \mathbb{L}$:\\[2mm]
\indent
(1) \ if $\alpha \in \Gamma$ then $\Gamma
\vdash_{\mathcal{L}} \alpha$;\\[1mm]
\indent
(2) \ if $\Gamma \vdash_{\mathcal{L}} \alpha$
and $\Gamma \subseteq \Delta$ then $\Delta \vdash_{\mathcal{L}} \alpha$;\\[1mm]
\indent
(3) \ if $\Delta \vdash_{\mathcal{L}} \alpha$ and
$\Gamma \vdash_{\mathcal{L}} \beta$ for every $\beta \in \Delta$ then $\Gamma
\vdash_{\mathcal{L}} \alpha$.\\

A Tarskian logic $\mathcal{L}$ is  {\em finitary} if it satisfies the following property: 
\\[2mm]
\indent (4) \ if $\Gamma \vdash_{\mathcal{L}} \alpha$ then there
exists  a finite subset $\Gamma_0$ of $\Gamma$ such that $\Gamma_0
\vdash_{\mathcal{L}} \alpha$.\\

Given a Tarskian logic  $\mathcal{L}$ over $\mathbb{L}$, let $\Gamma\cup\{\varphi\} \subseteq \mathbb{L}$. Then $\Gamma$ is {\em maximal non-trivial  w.r.t. $\varphi$} (or {\em $\varphi$-saturated}) in $\mathcal{L}$ if $\Gamma \nvdash_{\mathcal{L}} \varphi$ but $\Gamma, \psi \vdash_{\mathcal{L}} \varphi$ for any
        $\psi\not\in\Gamma$.

It is easy to see that if $\Gamma$ is $\varphi$-saturated then it is closed, that is, the following holds for every formula $\psi$: $\Gamma \vdash_{\mathcal{L}} \psi$ iff $\psi\in\Gamma$.

\begin{theorem}[Lindenbaum-\L os]\label{L-A-lemma}
Let $\mathcal{L}$ be a Tarskian and finitary logic over the language
$\mathbb{L}$. Let $\Gamma\cup \{\varphi\} \subseteq \mathbb{L}$ be
such that $\Gamma\nvdash_{\mathcal{L}} \varphi$.  There exists then  a set
$\Delta$ such that $\Gamma \subseteq \Delta \subseteq
\mathbb{L}$ and $\Delta$ is
$\varphi$-saturated in $\mathcal{L}$.
\end{theorem}
\begin{proof}
See~\cite[Theorem~22.2]{woj:84}. Another proof can be found in~\cite[Theorem~2.2.6]{CC16}.
\end{proof}

\

Since {\bf Km} is Tarskian and finitary (because it is defined by means of a Hilbert calculus with \MP\ as the only inference rule, which is finitary) then the last theorem holds for {\bf Km}.
 
\begin{definition} \label{Km-can}
  Let $\Delta$ be a set of formulas which is $\varphi$-saturated in {\bf Km}. The canonical valuation associated to $\Delta$ is the  function 
	$\mathbb{V}_\Delta: For \longrightarrow \{T^+,C^+,F^+,I^+,T^-,C^-,F^-,I^-\}$
	defined as follows:
  \begin{enumerate}
    \item $\mathbb{V}_\Delta(\alpha) \in {+}$ \ iff \ $\alpha \in \Delta$
    \item $\mathbb{V}_\Delta(\alpha) \in {-}$ \ iff \ $\neg \alpha \in \Delta$
    \item $\mathbb{V}_\Delta(\alpha) \in T$ \ iff \ $\Box \alpha, \Diamond \alpha \in \Delta$
    \item $\mathbb{V}_\Delta(\alpha) \in C$ \ iff \ $\Diamond \alpha, \Diamond \neg \alpha  \in \Delta$
    \item $\mathbb{V}_\Delta(\alpha) \in F$ \ iff \ $\Box\neg  \alpha, \Diamond \neg \alpha \in \Delta$.
    \item $\mathbb{V}_\Delta(\alpha) \in I$ \ iff \ $\Box  \alpha, \Box \neg \alpha \in \Delta$.    
  \end{enumerate}
\end{definition}

\

Given that $\Delta$ is $\varphi$-saturated in {\bf Km} then, for any formula $\alpha$, exactly one formula in the set $\{\alpha, \neg\alpha\}$ belongs to $\Delta$. This shows, taking into account Lemma~\ref{A}, that $\mathbb{V}_\Delta$ is a well-defined function.

\begin{lemma} \ \label{Km-matrix}
     Let $\Delta$ be a set of formulas which is $\varphi$-saturated in {\bf Km} for some formula $\varphi$, and let $\mathbb{V}_\Delta$ be the canonical valuation as described
     in Definition~\ref{Km-can}. Then, $\mathbb{V}_\Delta$ is a {\bf Km}-valuation.
\end{lemma}
\begin{proof} \

\begin{itemize}

\item[] {\bf CASE 1}: $\alpha$ is $\neg \beta$.

\begin{enumerate}[(i)]
	\item	If $\mathbb{V}_\Delta(\beta) \in T$, then $\Box\beta, \Diamond \beta \in \Delta$.
     		By~\axDNu\ it follows that $\Box \neg \neg \beta \in \Delta$ and by Lemma~\ref{A}~(iv), it follows
     		that $\Diamond \neg \neg \beta \in \Delta$. Thus, $\mathbb{V}_\Delta(\neg\beta) \in F$.
	\item	If $\mathbb{V}_\Delta(\beta) \in C$ then $\Diamond \beta, \Diamond \neg \beta \in \Delta$. 
			By Lemma~\ref{A}~(iv), we obtain $\Diamond \neg \neg \beta$. Thus, $\mathbb{V}_\Delta(\neg\beta) \in C$.
	\item 	If $\mathbb{V}_\Delta(\beta) \in F$ then $\Box\neg \beta, \Diamond \neg \beta \in
     		\Delta$. Hence $\mathbb{V}_\Delta(\neg\beta) \in T$.
    \item 	If $\mathbb{V}_\Delta(\beta) \in I$, then $\Box \beta, \Box \neg \beta \in \Delta$.
			By~\axDNu\ we have $\Box \neg \neg \beta \in \Delta$. Therefore,~$\mathbb{V}_\Delta(\neg\beta)\in~I$. 
 	\item  If $\mathbb{V}_\Delta(\beta) \in {+}$ then $\beta \in \Delta$.
     		By {\bf PC}, we have $\neg \neg \beta \in \Delta$. Thus,
     		$\mathbb{V}_\Delta(\neg\beta) \in {-}$.
	\item	If $\mathbb{V}_\Delta(\beta) \in {-}$, then $\neg\beta \in \Delta$. 
			Therefore, $\mathbb{V}_\Delta(\neg\beta) \in {+}$.
\end{enumerate} 

\item[] {\bf CASE 2}: $\alpha$ is $\Box \beta$.

\begin{enumerate}[(i)]
	\item	If $\mathbb{V}_\Delta(\beta) \in T$, then $\Box \beta, \Diamond \beta \in \Delta$
			and $\mathbb{V}_\Delta(\Box\beta) \in {+}$.
	\item	If $\mathbb{V}_\Delta(\beta) \in C$, then $\Diamond \beta, \Diamond \neg \beta \in
     		\Delta$. By Lemma~\ref{A}~(i), $\neg\Box \beta \in \Delta$ and
     		$\mathbb{V}_\Delta(\Box \beta) \in {-}$.
	\item 	If $\mathbb{V}_\Delta(\beta) \in F$, then $\Box \neg \beta, \Diamond \neg \beta \in \Delta$.
			As above, we have $\neg\Box \beta \in \Delta$ and
     		$\mathbb{V}_\Delta(\Box \beta) \in {-}$.
    \item If  $\mathbb{V}_\Delta(\beta) \in I$, then $\Box \beta, \Box \neg \beta \in \Delta$ and
    		$\mathbb{V}_\Delta(\Box \beta) \in {+}$.
\end{enumerate}
\end{itemize}
\newpage
\begin{itemize}
\item[] {\bf CASE 3}: $\alpha$ is $\beta \to \gamma$.

\begin{enumerate}[(i)]
	\item 	Consider that $\mathbb{V}_\Delta(\beta) \in I$. Thus, $\Box \beta, \Box \neg \beta \in \Delta$.
			Hence, by \axIu\  and {\bf PC}, we obtain $\Box (\beta \to \gamma),
			\Box \neg (\beta \to\gamma) \in \Delta$. Therefore, 
			$\mathbb{V}_\Delta(\beta \to \gamma) \in I$.
	\item 	Consider now that $\mathbb{V}_\Delta(\gamma) \in I$. Then, $\Box \gamma, \Box \neg \gamma \in \Delta$.
			Hence, by \axId\  and {\bf PC}, we have $\Box (\beta \to \gamma),
			\Box \neg (\beta \to \gamma) \in \Delta$. Thus, 
			$\mathbb{V}_\Delta(\beta \to \gamma)\in~I$.		
	\item 	If $\mathbb{V}_\Delta(\beta) \in T$ and $\mathbb{V}_\Delta(\gamma) \in T$, then
		  	$\Box \beta, \Diamond \beta, \Box \gamma, \Diamond \gamma \in \Delta$. Thus,
		  	by Lemma~\ref{B}~(iv), $\Diamond(\beta \to \gamma) \in \Delta$. Furthermore,
		  	by \axMd, $\Box(\beta \to \gamma) \in \Delta$. Therefore,
		  	$\mathbb{V}_\Delta(\beta \to \gamma) \in T$
	\item 	Suppose that $\mathbb{V}_\Delta(\beta) \in T$ and $\mathbb{V}_\Delta(\gamma) \in C$. Hence,
			$\Box \beta, \Diamond \beta, \Diamond \gamma, \Diamond \neg \gamma \in \Delta$. Thus, 
			by Lemma~\ref{B}~(iv), $\Diamond(\beta \to \gamma) \in \Delta$. Besides, by
			Lemma~\ref{B}~(i), $\Diamond \neg (\beta \to \gamma) \in \Delta$. Therefore,
			$\mathbb{V}_\Delta(\beta \to \gamma) \in C$
	\item  	Suppose now that $\mathbb{V}_\Delta(\beta) \in T$ and $\mathbb{V}_\Delta(\gamma) \in F$. Hence,
			$\Box \beta, \Diamond \beta,\Box \neg \gamma,\Diamond \neg \gamma \in \Delta$. Thus,
			by Lemma~\ref{B}~(ii), $\Diamond \neg (\beta \to \gamma)$. Furthermore, by Lemma~\ref{B}~(iii), we have
			$\Box \neg (\beta \to \gamma)$. Therefore, 
			$\mathbb{V}_\Delta(\beta \to \gamma) \in F$.
	\item   Let $\mathbb{V}_\Delta(\beta) \in C$ and $\mathbb{V}_\Delta(\gamma) \in T$. Hence, 
			$\Diamond \beta, \Diamond \neg \beta, \Box \gamma, \Diamond\gamma \in \Delta$. Thus, by \axMd, we 
			obtain $\Box (\beta \to \gamma) \in \Delta$. Besides, by Lemma~\ref{B}~(iv), we have
			$\Diamond (\beta \to \gamma) \in \Delta$. Therefore, 
			$\mathbb{V}_\Delta(\beta \to \gamma) \in T$.
	\item	Let now $\mathbb{V}_\Delta(\beta) \in C$ and $\mathbb{V}_\Delta(\gamma) \in C$. Thus, 
			$\Diamond \beta, \Diamond \neg \beta, \Diamond \gamma, \Diamond \neg \gamma \in \Delta$.
			Hence, by Lemma~\ref{B}~(iv), $\Diamond (\beta \to \gamma) \in \Delta$. Therefore,
			$\mathbb{V}_\Delta(\beta \to \gamma) \in T$ or 
			$\mathbb{V}_\Delta(\beta \to \gamma) \in C$.
	\item 	Suppose now that $\mathbb{V}_\Delta(\beta) \in C$ and $\mathbb{V}_\Delta(\gamma) \in F$. Hence, 
			$\Diamond \beta, \Diamond \neg \beta, \Box \neg \gamma, \Diamond \neg \gamma \in \Delta$.
			Thus, by Lemma~\ref{B}~(vi), we have $\Diamond (\beta \to \gamma) \in \Delta$. And by
			Lemma~\ref{B}~(ii), we have $\Diamond \neg (\beta \to \gamma)$. Therefore,
			$\mathbb{V}_\Delta(\beta \to \gamma) \in C$.
	\item 	Consider now $\mathbb{V}_\Delta(\beta) \in F$ and $\mathbb{V}_\Delta(\gamma) \in T$. Thus,
			$\Box \neg \beta, \Diamond \neg \beta, \Box \gamma, \Diamond \gamma \in \Delta$.
			On the one hand, by \axMd\ we have $\Box(\beta \to \gamma) \in \Delta$ . On the other
			hand,  we obtain $\Diamond(\beta \to \gamma) \in \Delta$, by Lemma~\ref{B}~(v).
			Therefore, $\mathbb{V}_\Delta(\beta \to \gamma) \in T$.
	\item   Let also $\mathbb{V}_\Delta(\beta) \in F$ and $\mathbb{V}_\Delta(\gamma) \in C$. Thus,
			$\Box \neg \beta, \Diamond \neg \beta, \Diamond \gamma, \Diamond \neg \gamma \in \Delta$.
			By Lemma~\ref{A}~(iii) and \axMu, we have $\Box(\beta \to \gamma) \in \Delta$. Furthermore,
			by Lemma~\ref{B}~(vi), we obtain $\Diamond(\beta \to \gamma) \in \Delta$. Therefore,
			$\mathbb{V}_\Delta(\beta \to \gamma) \in T$.
	\item	Finally, let $\mathbb{V}_\Delta(\beta) \in F$ and $\mathbb{V}_\Delta(\gamma) \in F$. Hence,
			$\Box \beta, \Diamond \neg \beta, \Box \neg \gamma, \Diamond \neg \gamma \in \Delta$.
			As above, by Lemma~\ref{A}~(iii), \axMu\ and Lemma~\ref{B}~(vi), we conclude that
			$\mathbb{V}_\Delta(\beta \to \gamma) \in T$.
	   
\end{enumerate} 
\end{itemize}
\end{proof}

\begin{corollary} \label{Km-true}\
Let $\Delta$ be a set of formulas which is $\varphi$-saturated in {\bf Km} for some formula $\varphi$. Then, the funcion $\mathbb{V}_\Delta$
is a  {\bf Km}-valuation such that, for every formula $\alpha$:
$$\mathbb{V}_\Delta(\alpha) \in + \ \textrm{ iff } \ \alpha \in \Delta.$$
\end{corollary}
\begin{proof}
By Definition~\ref{Km-can}, $\mathbb{V}_\Delta$
is a function such that $\mathbb{V}_\Delta(\alpha) \in +$ iff
$\alpha \in \Delta$. By Lemma~\ref{Km-matrix}, $\mathbb{V}_\Delta$
is a {\bf Km}-valuation. 
\end{proof}

\

\begin{theorem} [Completeness of {\bf Km}] \label{Km-compl} Let $\Gamma \cup\{\alpha\}$ be a set of formulas in $For$. Then:
$$\Gamma\vDash_{\bf Km} \alpha \ \mbox{ implies } \ \Gamma\vdash_{\bf Km} \alpha.$$
\end{theorem}
\begin{proof}
Suppose that $\Gamma\nvdash_{\bf Km} \alpha$.  By Theorem~\ref{L-A-lemma} there exists a set $\Delta$ of formulas such that $\Gamma \subseteq \Delta$ and $\Delta$ is
$\alpha$-saturated in {\bf Km}. Then, $\alpha \not\in \Delta$. By Corollary~\ref{Km-true}, $\mathbb{V}_\Delta$ is a {\bf Km}-valuation such that $\mathbb{V}_\Delta(\gamma) \in +$ for every $\gamma \in \Gamma$, but $\mathbb{V}_\Delta(\alpha) \not\in +$. Therefore $\Gamma \nvDash_{\bf Km} \alpha$.
\end{proof}

\

\begin{corollary} [Compactness of Nmatrix semantics for {\bf Km}] \label{Km-scompl} 
Let $\Gamma \cup\{\alpha\}$ be a set of formulas in $For$. If $\Gamma \vDash_{\bf Km} \alpha$ then there exists  a finite subset $\Gamma_0$ of $\Gamma$ such that $\Gamma_0
\vDash_{\bf Km} \alpha$.
\end{corollary}
\begin{proof}
It is a consequence of Theorems~\ref{Km-compl}  and~\ref{SoundKm}, and the fact that $\vdash_{\bf Km}$ is finitary.
\end{proof}

\

Let {\bf K4m} be the system obtained by adding to {\bf Km} the schema axiom \axfour. It is very easy to check that this system is sound and complete with respect to {\bf K4m}-Nmatrices. For soundness, we just observe that if $v(\Box \alpha) \in +$ then $v(\Box \alpha) \in \{T^+, I^+\}$ and thus $v(\Box\Box \alpha) \in +$. For completeness, note that: (i) if $\mathbb{V}_\Delta(\beta) \in T$, then $\Box \beta, \Diamond \beta \in \Delta$ and then, by \axfour, $\Box \Box \beta \in \Delta$
and  $\mathbb{V}_\Delta(\Box \beta) \in \{T^+,I^+\}$; if (ii) $\mathbb{V}_\Delta(\beta) \in I$, then $\Box \beta, \Box \neg \beta \in \Delta$ and, once again by \axfour, $\Box \Box \beta \in \Delta$. Hence $\mathbb{V}_\Delta(\Box \beta) \in \{T^+,I^+\}$ too.

Consider now the  {\bf K45m} system obtained by adding to {\bf K4m} the schema axiom \axfive. It is not difficult  to check that this system is sound and complete with respect to {\bf K45m}-Nmatrices. For soundness, note that if $v(\Diamond \Box \alpha) \in +$ then $v(\Box \alpha) \in T \cup C$ and then $v(\alpha)\subseteq T \cup I$, that is, $v(\alpha) \in \{T^+,I^+\}$. For completeness, note that: (i) if $\mathbb{V}_\Delta(\beta) \in C$, then $\Diamond \beta, \Diamond \neg \beta \in \Delta$. Hence, by \axfive\ and Lemma~\ref{A}~(i), $\Box \neg \Box \beta, \neg \Box \beta \in \Delta$ and  $\mathbb{V}_\Delta(\Box \beta) \in \{F^-,I^-\}$; if (ii) $\mathbb{V}_\Delta(\beta) \in F$, then $\Box \neg \beta, \Diamond \neg \beta \in \Delta$ and, once again by \axfive\ and Lemma~\ref{A}~(i), $\Box \neg \Box \beta, \neg \Box \beta \in \Delta$ and  $\mathbb{V}_\Delta(\Box \beta) \in \{F^-,I^-\}$ too.

\section{Recovering a logic inside another} \label{DAT}

The so-called {\em  recovery operators} play a fundamental role in the class of paraconsistent logics known as {\em Logics of Formal Inconsistency} (in short, {\bf LFI}s, see for instance~\cite{CC16}). Recall that a given logic {\bf S} is an {\bf LFI} if it is paraconsistent w.r.t. some negation $\neg$ (that is, there are formulas $\alpha$ and $\beta$ such that $\alpha,\neg\alpha \nvdash_{\bf S} \beta$). In addition, there is a  (primitive of defined) unary connective $\circ$ in {\bf S}, which is called a {\em consistency operator}, such that $\alpha,\neg\alpha, \circ\alpha \vdash_{\bf S} \beta$ for every formula $\alpha$ and $\beta$.\footnote{This is a slightly simplified version of the definition of {\bf LFI}s, see~\cite[Chapter~2]{CC16}.}  Given a logic {\bf S} which is  an  {\bf LFI} contained in a presentation of {\bf PC} in the same signature of {\bf S}, then the consistency operator $\circ$ allows us to recover {\bf PC} inside  {\bf S} as follows: for every (finite) set $\Gamma \cup \{\alpha\}$ of formulas,
$$\Gamma \vdash_{\bf PC} \alpha \ \mbox{ iff } \ (\exists \Upsilon)[\Gamma, \{\circ \beta \ : \  \beta \in \Upsilon\} \vdash_{\bf S} \alpha].$$
The relation stated above is called  a {\em Derivability Adjustment Theorem} (DAT).
The technique of DATs was originally proposed by D. Batens in the framework of {\em Adaptive logics}, but this idea, together with the notion  of consistency operators, was already used by N. da Costa for his well-known hierarchy of paraconsistent systems called $C_n$ (see~\cite{dac:63}). The basic {\bf LFI} called {\bf mbC} has the following axiom (called {\em gently explosion law}): $\circ\alpha \to(\alpha \to(\neg\alpha \to\beta))$. This axiom, together with the others,  guarantees the existence of a DAT between {\bf mbC} and  {\bf PC}. Observe that the standard explosion law $\alpha \to(\neg\alpha \to\beta)$ of  {\bf PC} is {\em recovered} in {\bf mbC} by adding the additional assumption $\circ\alpha$ (namely, `$\alpha$ is consistent'). Namely, $\circ\alpha \vdash_{\bf mbC} \alpha \to(\neg\alpha \to\beta)$.

It is worth noting that the system {\bf Km} was obtained from {\bf Dm} by weakening some of its axioms, by using a technique similar to that of {\bf LFI}s. Specifically, axioms  \axK, \axKu, \axKd, \axMt\ and \axMc\ of {\bf Dm} can only be applied in  {\bf Km} under certain assumptions, such as $\Diamond\alpha$ or $\Diamond\neg\beta$. This suggest the possibility of defining a formula which plays the role of a recovery operator in   {\bf Km} with respect to {\bf Dm}.

\begin{definition}
For any formula $\alpha$, let $\circ\alpha:=\Box\alpha \to \Diamond\alpha$ be the  {\em recovery operator for {\bf Km} with respect to {\bf Dm}}.
\end{definition}

\noindent Observe that the formula schema $\circ\alpha$ is $D(\alpha)$, the formula schema of the axiom schema \axD. It is also worth noting that the formula $\circ\alpha$ is equivalent to $\Diamond\alpha \vee \Diamond\neg\alpha$ in {\bf Km}. The multioperator associated to $\circ$ in the Nmatrix for {\bf Km} is defined as follows:

\begin{displaymath}
\begin{array}{|l|l|}
   	   \hline  	   & \circ \\
\hline \hline  T^+ & + \\
 		\hline C^+ & + \\
 		\hline F^+ & + \\
 		\hline I^+ & - \\
 		\hline T^- & + \\
 		\hline C^- & + \\
 		\hline F^- & + \\
 		\hline I^- & - \\
 		\hline	
\end{array}
\end{displaymath}

\ 

\noindent
That is, $\circ\alpha$ receives a designated truth-value if and only if $\alpha$ receives a value out of $I=\{I^+, I^-\}$; that is, if and only if $\alpha$ receives a truth-value  in the Nmatrix for {\bf Dm}. Boths aspects, namely: the syntactical one, concerning $\circ$ as being a recovery operator, and the semantical one, in which $\circ$ `marks out' exactly the truth-values in {\bf Dm}, will be fundamental in order to obtain a DAT for {\bf Km} with respect to {\bf Dm} (see Theorem~\ref{DAT1} below). If we consider the formula  $\circ\alpha$ as denoting the `modal consistency' of $\alpha$, then  $\bullet\alpha:=\neg{\circ}\alpha$ would be a formula denoting the `modal inconsistency' of $\alpha$. This terminology is suggested by the theory of Logics of Formal Inconsistency. Observe that  $\bullet\alpha$ is equivalent in {\bf Km} to  $\Box\alpha \wedge \Box \neg \alpha$, the formula used in axioms \axIu\ and \axId.  Clearly, the truth-value of $\bullet\alpha$ is designated if and only if the truth-value of $\alpha$ is in $\{I^+, I^-\}$.

In order to obtain a DAT between {\bf Km} and {\bf Dm}, it is necessary to guarantee that $\circ\alpha$ recovers all the axioms from {\bf Dm} which were weakened in  {\bf Km}. Recall the notation introduced in Notation~\ref{nota}.

\begin{lemma} \label{lema-DAT}
Let $\alpha$ be a formula. Then:\\[1mm]
(1) $\vdash_{\bf Km}  \Diamond\alpha \to \circ\alpha$;\\[1mm]
(2) $\vdash_{\bf Km} \Diamond\neg\alpha \to \circ\alpha$;\\[1mm]
(3) $\vdash_{\bf Km} (\Diamond\alpha \vee \Diamond\neg\alpha) \to \circ\alpha$;\\[1mm]
(4) $\vdash_{\bf Km} (\circ\alpha  \to K(\alpha,\beta)) \to K'(\alpha,\beta)$;\\[1mm]
(5) $\vdash_{\bf Km}  (\circ\beta  \to K1(\alpha,\beta)) \to K1'(\beta,\alpha)$;\\[1mm]
(6) $\vdash_{\bf Km}  (\circ\alpha  \to K2(\alpha,\beta)) \to K2'(\alpha,\beta)$;\\[1mm]
(7) $\vdash_{\bf Km}  (\circ\alpha  \to M3(\beta,\alpha)) \to M3'(\alpha,\beta)$;\\[1mm]
(8) $\vdash_{\bf Km}  (\circ\beta  \to M4(\alpha,\beta)) \to M4'(\beta,\alpha)$;\\[1mm]
(9) $(\bullet\alpha  \to \bullet(\alpha \to \beta)) \equiv I1(\alpha,\beta)$;\\[1mm]
(10) $(\bullet\beta  \to \bullet(\alpha \to \beta)) \equiv I2(\beta,\alpha)$.
\end{lemma}
\begin{proof} \ \\
(1)-(3): It follows from the fact that $\circ\alpha \equiv (\Diamond\alpha \vee \Diamond\neg\alpha)$  and by reasoning in  {\bf PC}.\\[1mm]
(4)-(8): It follows from (1)-(3) and by reasoning in  {\bf PC}.\\[1mm]
(9)-(10): It follows from the fact that $\bullet\alpha \equiv (\Box \alpha \wedge \Box \neg \alpha)$  and by reasoning in  {\bf PC}.
\end{proof}

Let ${\bf Km}_\circ$ be the system obtained from {\bf Km} by replacing the axiom schemas \axKp, \axKup, \axKdp, \axMtp, \axMcp, \axIu\ and \axId\ by the following seven axiom schemas:

$$
\begin{array}{ll}
\axKpp & \circ  \alpha \to K(\alpha,\beta)\\[2mm]
\axKupp & \circ \beta \to K1(\alpha,\beta)\\[2mm]
\axKdpp & \circ \alpha \to  K2(\alpha,\beta)\\[2mm]
\axMtpp & \circ \alpha \to M3(\beta,\alpha)\\[2mm]
\axMcpp & \circ\beta \to  M4(\alpha,\beta)\\[2mm]
\axIup & \bullet\alpha \to \bullet(\alpha \to \beta)\\[2mm]
\axIdp & \bullet\beta \to \bullet(\alpha \to \beta)
\end{array}$$

\begin{proposition} \label{equivKm}
The systems {\bf Km} and ${\bf Km}_\circ$ coincide, that is: $\Gamma \vdash_{\bf Km} \alpha$ iff $\Gamma \vdash_{{\bf Km}_\circ} \alpha$ for every $\Gamma \cup \{\alpha\} \subseteq For$.
\end{proposition}
\begin{proof}
By Lemma~\ref{lema-DAT}~(4)-(10) it follows that ${\bf Km}_\circ$ is stronger than {\bf Km}, that is: if $\Gamma \vdash_{\bf Km} \alpha$ then $\Gamma \vdash_{{\bf Km}_\circ} \alpha$. Conversely, by adapting the proof of Theorem~\ref{SoundKm} it follows that every axiom in ${\bf Km}_\circ$ is valid in the Nmatrix of {\bf Km}. By completeness of {\bf Km}, all these axioms are derivable in {\bf Km}. Hence, $\Gamma \vdash_{{\bf Km}_\circ} \alpha$ implies that $\Gamma \vdash_{\bf Km} \alpha$.  
\end{proof}

\

\begin{remark} {\em 
The latter result clearly shows that {\bf Km} can be seen as a weak version of {\bf Dm} in which some axioms (or principles) are weakened and controlled by means of a modal consistency recovery operator $\circ$, such that the corresponding modal inconsistency operator $\bullet$ satisfies certain propagation properties  (described by axioms \axIup\ and \axIdp). This is analogous to the case of {\bf LFI}s, as explained above. Observe that  \axIup\ and \axIdp\ are equivalent in {\bf Km} to
$$
\begin{array}{ll}
\axIupp & \circ(\alpha \to \beta) \to \circ\alpha\\[2mm]
\axIdpp & \circ(\alpha \to \beta) \to \circ\beta
\end{array}$$
which are properties of retropropagation of the modal consistency operator $\circ$. }
\end{remark}

Now, a simple but relevant  technical result concerning Nmatrices will be stated. Recall from~\cite{CFG} the notion of submultialgebra. It states that a multialgebra $\mathcal{A}$ is a submultialgebra of another $\mathcal{B}$, denoted by $\mathcal{A} \subseteq_{sm}\mathcal{B}$ provided that: both are defined over the same signature; the domain $A$ of $\mathcal{A}$ is contained in the domain $B$ of $\mathcal{B}$; and $c^\mathcal{A}(\vec a) \subseteq c^\mathcal{B}(\vec a)$ for every $n$-ary connective $c$ and any $\vec a \in A^n$. Here,  $c^\mathcal{A}$ and $c^\mathcal{B}$ denote the multioperations associated to the connective $c$ in $\mathcal{A}$ and $\mathcal{B}$, respectively.

\begin{proposition} \label{prop-subalg}
Let $\mathcal{A}_{\bf Tm}$, $\mathcal{A}_{\bf Dm}$ and $\mathcal{A}_{\bf Km}$ be the multialgebras underlying the Nmatrices  $\mathcal{M}_{\bf Tm}$, $\mathcal{M}_{\bf Dm}$  and $\mathcal{M}_{\bf Km}$  for the logics {\bf Tm}, {\bf Dm} and {\bf Km}, respectively. Then, the following holds: $\mathcal{A}_{\bf Tm} \subseteq_{sm} \mathcal{A}_{\bf Dm} \subseteq_{sm}\mathcal{A}_{\bf Km}$.   
\end{proposition}
\begin{proof}
If follows by mere inspection of the multioperations for each logic.
\end{proof}

\

\begin{corollary} \label{val-restr} Let $A_{\bf Tm}$, $A_{\bf Dm}$ and $A_{\bf Km}$ be the domains of the Nmatrices  $\mathcal{M}_{\bf Tm}$, $\mathcal{M}_{\bf Dm}$  and $\mathcal{M}_{\bf Km}$  for the logics {\bf Tm}, {\bf Dm} and {\bf Km}, respectively.\\
1) Let $v$ ve a valuation in  $\mathcal{M}_{\bf Tm}$, and let $v_1:For \to A_{\bf Dm}$ and $v_2:For \to A_{\bf Km}$ be the functions given by $v_1(\alpha)=v_2(\alpha)=v(\alpha)$, for every formula $\alpha$. Then $v_1$ and $v_2$ are valuations in $\mathcal{M}_{\bf Dm}$ and $\mathcal{M}_{\bf Km}$, respectively.\\
2) Let $v$ ve a valuation in  $\mathcal{M}_{\bf Dm}$, and let $v_1:For \to  A_{\bf Km}$ be the function given by $v_1(\alpha)=v(\alpha)$, for every formula $\alpha$. Then $v_1$ is a valuation in $\mathcal{M}_{\bf Km}$.
\end{corollary}
\begin{proof}
It is immediate from the definition of valuation over an Nmatrix, and from Proposition~\ref{prop-subalg}.
\end{proof}

\

\begin{theorem} [DAT between {\bf Km} and {\bf Dm}] \label{DAT1}
Let $\Gamma \cup \{\varphi\} \subseteq For$ be a set of formulas. Then:
$$\Gamma \vdash_{\bf Dm} \varphi \ \mbox{ iff } \ (\exists \Upsilon)[\Gamma, \{\circ \delta \ : \  \delta \in \Upsilon\} \vdash_{\bf Km} \varphi].$$
\end{theorem}
\begin{proof}
({\em `Only if'}  part)
We begin by proving the following (recalling Notation~\ref{nota}): \\[1mm]
{\bf Fact 1:} Let $\textsf{(A)} \in \{\axK, \axKd\}$. Then $\circ\alpha \vdash_{\bf Km} A(\alpha,\beta)$ for every formula $\alpha,\beta$.\\[1mm]
Indeed, suppose that \textsf{(A)} is \axK. By Proposition~\ref{equivKm} the instance $\circ\alpha \to K(\alpha,\beta)$ of axiom \axKpp\ is derivable in {\bf Km}. Let $\alpha_1\ldots\alpha_n=\circ\alpha \to K(\alpha,\beta)$ be a derivation of that formula in {\bf Km}. Consider the following derivation in {\bf Km}:
$$
\begin{array}{ll}
1. & \circ\alpha \ \mbox{(hyp)}\\
2. & \alpha_1\\
& \vdots\\
n+1. & \alpha_n=\circ\alpha \to K(\alpha,\beta)\\ 
n+2. & K(\alpha,\beta) \ \mbox{(\MP\ $1$, $n+1$)} 
\end{array}$$
The above derivation shows that $\circ\alpha \vdash_{\bf Km} K(\alpha,\beta)$. The proof for axiom \axKd\ is similar, but now by using axiom \axKdpp.\\[1mm]
{\bf Fact 2:} Let $\textsf{(A)} \in \{\axKu, \axMc\}$. Then $\circ\beta \vdash_{\bf Km} A(\alpha,\beta)$ for every formula $\alpha,\beta$.\\[1mm]
The proof of {\bf Fact 2} is analogous to the one given for {\bf Fact 1}, but now by using axioms \axKupp\ and \axMcpp.\\[1mm]
{\bf Fact 3:} $\circ\alpha \vdash_{\bf Km} M3(\beta,\alpha)$ for every formula $\alpha,\beta$.\\[1mm]
The proof of {\bf Fact 3} is similar to that for {\bf Fact 1}, but now by using axiom \axMtpp.\\[2mm]
Now,  suppose that $\Gamma \vdash_{\bf Dm} \varphi$. Then, there is a finite sequence of formulas $\alpha_1\cdots \alpha_n$ such that $\alpha_n=\varphi$ and, for every $1 \leq i \leq n$, either $\alpha_i \in \Gamma$, or $\alpha_i$ is an instance of an axiom of {\bf Dm}, or there are $j,k<i$ such that $\alpha_i$ is obtained from $\alpha_j$ and $\alpha_k$ by \MP. In case that every instance of an axiom used in such derivation corresponds to axioms of {\bf Km}, then $\Gamma \vdash_{\bf Km} \varphi$ and so it suffices to take $\Upsilon=\emptyset$. Otherwise, let $\alpha_i$ be a formula in the derivation above which was included only because it is an instance of an axiom $\textsf{(A)}$ in {\bf Dm}, and assume that $\textsf{(A)}$ is not an axiom in {\bf Km}. That occurrence of $\alpha_i$ can be replaced in the sequence $\alpha_1\cdots \alpha_n$ by: either an occurrence of $\circ\alpha$ (which is  $\alpha_i$) as an additional hypothesis, if $\textsf{(A)}$ is \axD; or by  a  derivation of $\alpha_i$ in {\bf Km} from the additional hypothesis $\circ\alpha$ or $\circ\beta$ as described in the proofs of {\bf Fact 1}-{\bf Fact 3}, otherwise. Let $\Upsilon$ be the set of formulas $\delta$ such that $\circ\delta$ was added as an additional hypothesis by the method described above. Then, the sequence of formulas $\beta_1\cdots \beta_m$ obtained from  $\alpha_1\cdots \alpha_n$ by means of this process constitutes a derivation in {\bf Km} of $\varphi$ from $\Gamma \cup \{\circ \delta \ : \  \delta \in \Upsilon\}$.\\[1mm]
({\em `If'}  part) Suppose that  $\Gamma, \{\circ \delta \ : \  \delta \in \Upsilon\} \vdash_{\bf Km} \varphi$ for some set $\Upsilon$ of formulas, and let $v$ be a valuation in  the Nmatrix $\mathcal{M}_{\bf Dm}$ for {\bf Dm} such that $v(\gamma)$ is designated, for every $\gamma  \in \Gamma$.  By  Corollary~\ref{val-restr}(2), the function $v_1:For \to  A_{\bf Km}$ given by $v_1(\alpha)=v(\alpha)$ for every $\alpha$ is a valuation in $\mathcal{M}_{\bf Km}$. Moreover, $v_1(\gamma)$ is designated in $\mathcal{M}_{\bf Km}$ for every $\gamma  \in \Gamma \cup \{\circ \delta \ : \  \delta \in \Upsilon\}$. This follows from the fact that $v_1(\gamma) \in  A_{\bf Dm}$ for every $\gamma$ and by the definition of the multioperator associated to $\circ$, as observed above. By hypothesis, and by soundness of {\bf Km}, $v_1(\varphi)$ is designated in  $\mathcal{M}_{\bf Km}$ and so, by definition of $v_1$ and of $\mathcal{M}_{\bf Dm}$, $v(\varphi)$ is designated in  $\mathcal{M}_{\bf Dm}$. This shows that  $\Gamma \models_{\bf Dm} \varphi$. Thus, $\Gamma \vdash_{\bf Dm} \varphi$ by completeness of {\bf Km} (see Theorem~\ref{Km-compl}).
\end{proof}

\

Now, consider the formula schema $\circ'\alpha:=(\Box\alpha \to \alpha) \wedge (\Box\neg\alpha \to \neg\alpha)$ (recalling that $\beta \wedge \gamma$ is a notation for $\neg(\beta \to \neg\gamma)$). With the notation for axioms stated in Notation~\ref{nota}, it follows that $\circ'\alpha$ corresponds to $T(\alpha) \wedge T(\neg\alpha)$, where $T$ is the formula schema of axiom \axT. In addition, observe that,  in $\mathcal{A}_{\bf Dm}$, $\circ' x=\{T^-, C^-\}$ if $x \in \{F^+, T^-\}$, and $\circ' x={+}$ otherwise. That is,  $v(\circ'\alpha)$ is a designated value in $\mathcal{M}_{\bf Dm}$ iff $v(\alpha)$ is a value in $A_{\bf Tm}$, for every valuation $v$ for $\mathcal{M}_{\bf Dm}$.  This being so, and by a technique analogous to that used in the proof of  Theorem~\ref{DAT1}, the following DATs can be easily proved:

\begin{theorem} [DAT between {\bf Dm} and {\bf Tm}] \label{DAT2}
Let $\Gamma \cup \{\varphi\} \subseteq For$ be a finite set of formulas. Then:
$$\Gamma \vdash_{\bf Tm} \varphi \ \mbox{ iff } \ (\exists \Upsilon)[\Gamma, \{\circ' \gamma \ : \  \gamma \in \Upsilon\} \vdash_{\bf Dm} \varphi].$$
\end{theorem}

\

\begin{theorem} [DAT between {\bf Km} and {\bf Tm}] \label{DAT3}
Let $\Gamma \cup \{\varphi\} \subseteq For$ be a finite set of formulas. Then:
$$\Gamma \vdash_{\bf Tm} \varphi \ \mbox{ iff } \ (\exists \Upsilon,\Upsilon')[\Gamma, \{\circ \delta \ : \  \delta \in \Upsilon\}, \{\circ' \gamma \ : \  \gamma \in \Upsilon'\} \vdash_{\bf Km} \varphi].$$
\end{theorem}

\section{A deterministic implication} \label{sectDelta}

In 1996, M. Baaz expanded the infinite-valued G\"odel logic with  an unary operator called $\Delta$ defined as follows: $\Delta(1)=1$, and $\Delta(x)=0$ otherwise (see~\cite{baaz:96}).  The (dual of the) same operator was already considered by A. Monteiro, inspired by an example given in 1963 by L. Monteiro~\cite{LMont:63}, to consider in 1978 the variety $\mathbf{TMA}$ of {\em tetravalent modal algebras}, which expand the De Morgan algebras. This is why this kind of operator is now called {\em Monteiro-Baaz $\Delta$ operator}. In the context of fuzzy logic, P. H\'ajeck considered ${\bf BL}_\Delta$, the expansion of the basic fuzzy logic {\bf BL} by the $\Delta$ operator, whose axioms concerning $\Delta$ are the following (see~\cite[Chapter~2]{Hajeck}):

\begin{itemize}
  \item[] ($\Delta1$) \ \ $\Delta \alpha \vee \neg \Delta \alpha$
  \item[] ($\Delta2$) \ \ $\Delta (\alpha \vee \beta) \to (\Delta \alpha \vee \Delta \beta)$
  \item[] ($\Delta3$) \ \ $\Delta\alpha \to \alpha$
  \item[] ($\Delta4$) \ \ $\Delta\alpha \to \Delta\Delta\alpha$
  \item[] ($\Delta5$) \ \ $\Delta (\alpha \to \beta) \to (\Delta \alpha \to \Delta \beta)$
\end{itemize}

\

\noindent besides the  {\em Generalization} inference rule, analogous to necessitation rule in normal modal systems. As observed by H\'ajeck, the $\Delta$ operator is a kind of mixed modality which satisfies  properties of an (alethic) necessity operator, as well as a property, namely axiom ($\Delta2$), of a possibility operator.

In this section, we will consider $\rightarrow$ as a completely deterministic multioperator, under the perspective of interpreting $\Box$ as a kind of (non-deterministic) $\Delta$ operator. Because of axiom ($\Delta3$), it is natural of start with system {\bf Tm}. Given that $\alpha \vee \neg \alpha$ is equivalent in  {\bf Tm} to $\alpha \to \alpha$, then $\Box$ satisfies axiom ($\Delta1$). In addition, axiom ($\Delta5$) corresponds to axiom \axK, which is valid in    {\bf Tm}. However, neither ($\Delta2$) nor ($\Delta4$) are valid in   {\bf Tm}. Clearly  ($\Delta4$) corresponds to axiom \axfour, hence the $\Box$ operator together with the implication from system {\bf T4m} satisfy all the properties of a $\Delta$ operator, with the exception of ($\Delta2$).

In order to satisfy axiom ($\Delta2$), recall that $\alpha \vee \beta:= \neg \alpha \to \beta$. This being so, the implication should be deterministic in such a way that the induced disjunction becomes the join (maximum) operator in the chain $F^- \leq C^- \leq C^+ \leq T^+$. Namely,

\begin{displaymath}
\begin{array}{c c}
\begin{array}{|l|l|l|l|l|}
 \hline \to  & T^+ & C^+ & C^- 	& F^-\\
   \hline \hline T^+ & T^+ 	  & C^+ 	& C^- 	& F^- \\
 		  \hline C^+ & T^+ 	  & C^+ 	& C^- 	& C^- \\
 		  \hline C^- & T^+ 	  & C^+ 	& C^+ 	& C^+ \\
 		  \hline C^- & T^+ 	  & T^+ 	& T^+ 	& T^+ \\	
 		  \hline
\end{array}
&

\begin{array}{|l|l|l|l|l|}
 \hline \vee  & T^+ & C^+ & C^- 	& F^-\\
   \hline \hline T^+ & T^+ 	  & T^+ 	& T^+ 	& T^+ \\
 		  \hline C^+ & T^+ 	  & C^+ 	& C^+ 	& C^+ \\
 		  \hline C^- & T^+ 	  & C^+ 	& C^- 	& C^- \\
 		  \hline C^- & T^+ 	  & C^+ 	& C^- 	& F^- \\	
 		  \hline
\end{array}
\end{array}
\end{displaymath}

\

\noindent The Nmatrix obtained from the one for {\bf Tm} by changing the original multiperator for $\to$ by the (deterministic) operator above was already considered by Ivlev under the name of $S_{\bar a}$. He proposes the following additional axiom to characterize this system:

$$
\begin{array}{ll}
\Kdet & \Box (\alpha \to \beta) \to (\Diamond\alpha \to \Box\beta)
\end{array}$$

\begin{definition}
Let {\bf Tmd} be the Hilbert calculus obtained from {\bf Tm} by changing axiom {\em \axKu} by {\em \Kdet}.\footnote{The letter `d' stands for `deterministic implication/disjunction'.}
\end{definition}

\begin{proposition}
Axiom {\em \axKu}\ is derivable in  {\bf Tmd}.
\end{proposition}
\begin{proof}
It follows form the fact that the schema $\Box \alpha \to \Diamond\alpha$ is derivable in ${\bf PC} \cup \{\axT, \axDNu, \axDNd\}$ (see~\cite{con:cer:per:17}). 
\end{proof}

From the last result, it follows that {\bf Tmd} is equivalent to adding axiom \Kdet\ to  {\bf Tm}.

\begin{theorem} \label{Tmd-comp}
System {\bf Tmd} is sound and complete w.r.t. the Nmatrix for  {\bf Tm} in which the original multiperator for $\to$ is replaced by the deterministic operator of the truth-table above.\footnote{That is, the Nmatrix of Ivlev's  $S_{\bar a}$. See \cite{ivl:88}.} 
\end{theorem}

\begin{proof} 
For soundness, it worth noting that the Nmatrix for {\bf Tmd} restrains the  Nmatrix for {\bf Tm}. This means that any valuation over the Nmatrix for {\bf Tmd} is a valuation over the Nmatrix for {\bf Tm}. Hence, for {\bf Tm}-soundness proved in
 \cite{con:cer:per:15} and \cite{con:cer:per:17}, we have to check only the axiom \Kdet. Thus, let $\alpha,\beta \in For$ and let $v$ be a valuation  over the Nmatrix for {\bf Tmd} such that  $v(\Box(\alpha \to \beta))$ and $v(\Diamond \alpha)$ are designated. Then  $v(\alpha \to \beta) = T^+$ and $v(\alpha)  \neq F^-$. Thus, $v(\beta) = T^+$ and $(v \Box \beta)$ is designated. We conclude that $v(Kdet)$ is
designated too.

For completeness, note that Lemma~\ref{A} holds good for {\bf Tmd}. We have to make some adjustment in the proof of \cite[Lemma~4]{con:cer:per:15} by changing, for instance, the values $t^n$, $t^c$, $f^c$ and $f^i$ by $T^+$, $C^+$,$C^-$ and $F^-$, respectively. In addition, $\Delta$ is now an $\alpha$-saturated set containing $\Gamma$ such that $\Gamma \nvdash_{\bf Tmd} \alpha$.   Other adjustment that must be made is to eliminate the rule (DN), what is was  done in~\cite{con:cer:per:17}.
Now, three new cases must be considered:
\begin{enumerate}[(i)]
  \item If $\mathbb{V}_\Delta(\beta) = C^-$ and $\mathbb{V}_\Delta(\gamma) = C^+$, then $\Diamond \beta, \Diamond \neg \gamma, \neg \beta, \gamma \in \Delta$. If $\Box(\beta \to \gamma) \in \Delta$, by \Kdet\ and Lemma~\ref{A}~(i) we would have $\Box \gamma, \neg \Box \gamma \in \Delta$ and $\Delta$ would be inconsistent, an absurd. By {\bf PC}, Lemma~\ref{A}~(i) and the fact that $\Delta$ is $\alpha$-saturated, we have that $\mathbb{V}_\Delta(\beta \to \gamma)  = C^+$.
  \item If $\mathbb{V}_\Delta(\gamma) = \mathbb{V}_\Delta(\beta) = C^-$, then $\Diamond \beta, \Diamond \gamma, \neg \beta, \neg \gamma \in \Delta$.  If $\Box(\beta \to \gamma) \in \Delta$, by \Kdet\ and \axT\ we would have $\gamma \in \Delta$ and $\Delta$ would be inconsistent. By the same reasons, we have $\mathbb{V}_\Delta(\beta \to \gamma)  = C^+$.
    \item If  $\mathbb{V}_\Delta(\gamma) = \mathbb{V}_\Delta(\beta) = C^+$, then $\Diamond \neg \beta, \Diamond \neg \gamma, \beta, \gamma \in \Delta$. Suppose 
    $\Box(\beta \to \gamma) \in \Delta$. If $\neg \Diamond \beta \in \Delta$, by Lemma~\ref{A}~(iii)  and \axT, we would have $\neg \beta \in \Delta$ and $\Delta$
    would be inconsistent,  an absurd. Thus $\Diamond \beta \in \Delta$. In this case, by \Kdet\ and Lemma~\ref{A}~(i), we will have $\Box \gamma \in \Delta$. Again, by the same reasons as above, we conclude that $\mathbb{V}_\Delta(\beta \to \gamma)  = C^+$.
\end{enumerate}

\end{proof}

It is easy to see that axiom ($\Delta2$) holds in {\bf Tmd}, hence it satisfies all the axioms for a $\Delta$ operator with the exception of axiom ($\Delta4$). This justifies  the definition of the following system:

\begin{definition}
Let {\bf Tm4d} be the Hilbert calculus obtained from {\bf Tmd} by adding schema axiom {\em \axfour}. Equivalently, {\bf Tm4d} is obtained from {\bf Tm4} by changing axiom {\em \axKu} by {\em \Kdet}.
\end{definition}

\begin{theorem}
System {\bf Tm4d} is sound and complete w.r.t. the Nmatrix for  {\bf T4m} in which the original multiperator for $\to$ is replaced by the deterministic operator $\to$ .
\end{theorem} 

\begin{proof}
	We just have to add in Theorem \ref{Tmd-comp} the case when $\alpha =\Box \beta$. This case, in fact, was checked in \cite[Theorem~17]{con:cer:per:15}.
\end{proof}
 
Clearly, the logic {\bf Tm4d} satisfies all the axioms for a $\Delta$ operator. This being so, $\Box$ is a kind of non-deterministic $\Delta$ operator, in which  the necessation rule does not hold. It is worth noting that $\Box$ is the only non-deterministic connective of  {\bf Tm4d}. This means that the extension of {\bf Tm4d} obtained by adding schema axiom \axfive, namely, the system {\bf Tm45d}, is a system characterized by a standard logical matrix, since $\Box$ is also deterministic. By re-interpreting its truth-values $T^+$, $C^+$, $C^-$, $F^-$ as $1$, $2/3$, $1/3$ and $0$, respectively, and by observing that $\to$ and $\vee$ are inter-definable in {\bf Tm4d} (hence in {\bf Tm45d}), the logical matrix of  {\bf Tm45d} corresponds to the $\{\neg,\vee\}$-fragment of 4-valued \L ukasiewicz logic $\L4$ expanded with this operator simular (but weaker) to Monteiro-Baaz $\Delta$, and in which $\{1,2/3\}$ are designated. As we shall see in the next section, {\bf Tm45d} is the only modal system presented in this paper which is characterizable by a single finite deterministic matrix.

\section{The inevitability of non-determinism} \label{SectDugun}

In 1940, Dugundji showed that no modal system between {\bf S1} and {\bf S5} can be characterized by a single finite logical matrix. Since then, several generalizations of that theorem were obtained, showing that practically no modal system considered in the literature can be characterized by  a single finite logical matrix.  On the other hand, each of the systems considered in this paper (with exception of {\bf Tm45d}, mentioned at the end of the previous section) is characterized by a single finite-valued non-deterministic matrix. A natural question is to determine whether non-deterministic matrices are intrinsically necessary in order to characterize such systems, or if its possible to characterize (some of) them by means of a  single finite (ordinary) logical matrix.  

The aim of this section is showing that the above question has a negative answer. Indeed, two Dugundji-like theorems will be obtained (Theorems~\ref{Dug1} and~\ref{Dug2}),  proving that none of the systems considered in the previous sections (with exception of {\bf Tm45d}) can be characterized by a a single finite logical matrix. These results  justify the use of non-deterministic matrices as a semantic tool for dealing with them.

To begin with, some notation will be introduced. If $\Gamma=\{\alpha_1,\ldots, \alpha_n \}$ is a finite family of $n$ formulas in $For$, with $n\geq 2$, then $\bigvee \Gamma$ will denote the formula $((\ldots ((\alpha _1\vee \alpha _2) \vee \alpha _3) \vee \ldots) \vee \alpha_n)$.

\begin{definition}
Let $n \geq 3$ be a natural number, and let $p_1, \ldots p_n$  be $n$ different propositional variables. We define the following formula schema:
$$\delta(n) = \bigvee \{(\Box \alpha(n) \to \Box \beta_i(n))  \ : \ 1 \leq i \leq n\}$$
where $\alpha(n)=\bigvee \{ p_j \ : \ 1 \leq j \leq n\}$ and $\beta_i(n)=\bigvee\{ p_j \ : \ 1 \leq j \leq n, \ j\neq i\}$ for every $1 \leq i \leq n$.
\end{definition}

For instance, if $n=3$ then:

$$
\begin{array}{l l}
	\delta(3) = 	& \big((\Box ((p_1 \vee p_2) \vee p_3)\to \Box(p_2 \vee p_3)) \vee \\
					& (\Box ((p_1 \vee p_2) \vee p_3)\to \Box(p_1 \vee p_3))\big) \vee \\
					& (\Box ((p_1 \vee p_2) \vee p_3)\to \Box(p_1 \vee p_2)).
\end{array}
$$

\begin{proposition} \label{M-validates-delta} Let $\mathcal{M}=\langle M,D\rangle$ be a deterministic logical matrix such that $M$ has $n$ elements, and $\mathcal{M}$ is a model of {\bf Km}. Then, $\mathcal{M}$ validates the formula  $\delta(n+1)$.
\end{proposition}
\begin{proof}
Let $\mathcal{M}=\langle M,D\rangle$ be an $n$-valued  logical matrix which is a model of {\bf Km}, and let $h$ be a valuation over $\mathcal{M}$. Since $M$ has just $n$ values, there exists $1\leq i,k\leq n+1$ such that $i\neq k$ but $h(p_i) = h(p_k)$. From this, $h(\alpha(n+1))=h(\beta_i(n+1))$ since $h$ is a valuation over a logical matrix, and  $h(p_i) = h(p_k)$ by hypotesis (therefore, when evaluating  $h(\beta_i(n+1))$  the ``missing'' value  $h(p_i)$ is retrieved with $h(p_k)$). Using once again the fact that $h$ is a valuation over a  logical matrix, it follows that $h(\Box \alpha(n+1))= \Box h(\alpha(n+1))= \Box  h(\beta_i(n+1)) = h(\Box \beta_i(n+1))$. Thus, $h(\Box \alpha(n+1) \to \Box \beta_i(n+1)) = h(\Box \alpha(n+1)) \to h(\Box \beta_i(n+1))$ belongs to $D$, given that $\beta \to \beta$ is a theorem of {\bf Km} for every formula $\beta$. Using that $\beta \to (\beta \vee \gamma)$ and $\gamma \to (\beta \vee \gamma)$ are  theorems of {\bf Km} for every formula $\beta$ and $\gamma$, it follows that $h(\beta \to (\beta \vee \gamma)) \in D$ and $h(\gamma \to (\beta \vee \gamma)) \in D$. Hence, $h(\delta(n+1)) \in D$. Since the latter holds for every valuation $h$, the result follows.
\end{proof}

\

\begin{proposition} \label{prop-delta-not-valid}
The formula $\delta(n)$ is not a theorem of {\bf T45m}, for every  $n \geq 3$.
\end{proposition}
\begin{proof}
Let $v$ be a valuation over the Nmatrix for {\bf T45m} such that $v(p_j)=C^+$ for every $1 \leq j \leq n$; $v(\alpha(n))=T^+$, and $v(\beta_i(n))=C^+$ for every $1 \leq i \leq n$ (note that this is always possible, from the definition of the Nmatrix for {\bf T45m}). From this,  $v(\Box \alpha(n))$ is designated but $v(\Box \beta_i(n))$ is not, for every $i$. This being so, $v((\Box \alpha(n) \to \Box \beta_i(n))$ is not designated, for every $i$, hence $v(\delta(n))$ is not designated. This means that $\delta(n))$ is not a valid formula for the Nmatrix for  {\bf T45m} and so, by soundness, it is not a theorem of  {\bf T45m}.
\end{proof}

\

\noindent Now, the first theorem of uncharacterizability by finite matrices is obtained:

\begin{theorem} \label{Dug1}
  No system between {\bf Km} and {\bf T45m} can be characterized by a single finite deterministic logical matrix.
\end{theorem}
\begin{proof}
Suppose, by absurd, that a system  {\bf S} lying between {\bf Km} and {\bf T45m} can be semantically characterized by a finite logical matrix $\mathcal{M}$ with $n$ elements. It is worth noting that $n \geq 2$.  Indeed, if $n=1$ then  {\bf S} would be the trivial logic and so it could not be contained in {\bf T45m}.

Since $\mathcal{M}$  is a model of {\bf Km} then $\mathcal{M}$ validates the formula  $\delta(n+1)$, by Proposition~\ref{M-validates-delta}. Given that {\bf S} is complete w.r.t.  $\mathcal{M}$, the formula  $\delta(n+1)$ is a theorem of  {\bf S} and so it is a theorem of {\bf T45m}. But this contradicts Proposition~ \ref{prop-delta-not-valid}. Therefore, {\bf S}  cannot be characterized by a single finite logical matrix.
\end{proof}

\

\noindent The second theorem of uncharacterizability by finite matrices is obtained for the systems  {\bf Tmd} and {\bf T4md} introduced in Section~\ref{sectDelta}, in which the implication is deterministic. Firstly, some previous result are necessary:

\begin{definition}
Let $n \geq 3$ be a natural number, and let $p_1, \ldots p_n$  be $n$ different propositional variables. We define the following formula schema:
$$\gamma(n) = \bigvee \{(\Box \neg \Box\alpha(n) \to \Box \neg \Box \beta_i(n))  \ : \ 1 \leq i \leq n\}$$
where $\alpha(n)=\bigvee \{ p_j \ : \ 1 \leq j \leq n\}$ and $\beta_i(n)=\bigvee\{ p_j \ : \ 1 \leq j \leq n, \ j\neq i\}$ for every $1 \leq i \leq n$.
\end{definition}

\begin{proposition} \label{M-validates-gamma} Let ${\bf S} \in \{ {\bf Tmd}, {\bf T4md}\}$. Let $\mathcal{M}=\langle M,D\rangle$ be a (deterministic) logical matrix such that $M$ has $n$ elements, and $\mathcal{M}$ is a model of {\bf S}. Then, $\mathcal{M}$ validates the formula  $\gamma(n+1)$.
\end{proposition}
\begin{proof}
The proof is similar to that for Proposition~\ref{M-validates-delta}.
\end{proof}

\

\begin{proposition} \label{prop-gamma-not-valid}
Let ${\bf S} \in \{ {\bf Tmd}, {\bf T4md}\}$. Then, the  formula $\gamma(n)$ is not a theorem of {\bf S}, for every  $n \geq 3$.
\end{proposition}
\begin{proof}
Let $v$ be a valuation over the Nmatrix for {\bf S} such that $v(p_j)=C^+$ for every $1 \leq j \leq n$. Then, $v(\alpha(n))=v(\beta_i(n))=C^+$ for every $1 \leq i \leq n$, since the disjunction is deterministic. Now, assume that  $v(\Box \alpha(n))=F^-$ and $v(\Box \beta_i(n))=C^-$, for every $i$  (note that this is always possible, from the definition of the multioperator $\Box$ of the  Nmatrix for {\bf S}). From this, $v(\neg\Box \alpha(n))=T^+$ and $v(\Box \beta_i(n))=C^+$, for every $i$. This being so, $v(\Box \neg\Box \alpha(n))$ is designated but $v(\Box \neg\Box \beta_i(n))$ is not. Hence, $v((\Box \neg\Box \alpha(n) \to \Box \neg\Box \beta_i(n))$ is not designated, for every $i$, and so $v(\gamma(n))$ is not designated. From this, $\gamma(n))$ is not a valid formula for the Nmatrix for  {\bf S} and so, by soundness, it is not a theorem of  {\bf S}.
\end{proof}

\begin{theorem}  \label{Dug2}
Let ${\bf S} \in \{ {\bf Tmd}, {\bf T4md}\}$. Then, {\bf S} cannot be characterized by a single finite deterministic logical matrix.
\end{theorem}
\begin{proof}
Suppose, by absurd, that  {\bf S} can be semantically characterized by a finite logical matrix $\mathcal{M}$ with $n$ elements. As in the proof of Theorem~\ref{Dug1} it follows that $n \geq 2$.

By Proposition~\ref{M-validates-gamma} it follows that $\mathcal{M}$ validates the formula  $\gamma(n+1)$. Given that {\bf S} is complete w.r.t.  $\mathcal{M}$, the formula  $\gamma(n+1)$ is a theorem of  {\bf S}. But this contradicts Proposition~ \ref{prop-gamma-not-valid}. Therefore, {\bf S}  cannot be characterized by a single finite logical matrix.
\end{proof}

\

\noindent
Finally, it will be shown that all the modal systems introduced above are conservative extensions of  classical propositional logic {\bf PC}:

\begin{proposition} \label{cons1}
Every system {\bf S} between {\bf Km} and {\bf T45m} is a conservative extension of classical propositional logic {\bf PC}. That is, given a formula $\alpha$ over the signature $\{\neg,\to\}$, $\alpha$ is valid in  {\bf S} iff $\alpha$ is valid in {\bf PC}.
\end{proposition}
\begin{proof}
Given that {\bf S} is an extension of {\bf PC}, it follows that any valid formula of {\bf PC} is valid in {\bf S}. Now, suppose that $\alpha$ is not valid in {\bf PC}. Then, there is a 2-valued classical valuation $v$ such that $v(\alpha)=0$. It is clear that the 2-valued Boolean algebra for {\bf PC} can be defined with domain $\{T^+,F^-\}$ such that  $T^+$ corresponds to $1$ and $F^-$ corresponds to $0$. Moreover, this  algebra, seen as a multialgebra, is a submultialgebra of the Nmatrix of   {\bf T45m}, and its designated truth-value is also designated in  {\bf T45m}. From this, the valuation $v$ can be seen as a  valuation $\bar v$ over  the Nmatrix of   {\bf T45m} such that  $\bar v(\beta)$ is designated iff $v(\beta)=1$, for every formula $\beta$. Thus, the formula $\alpha$ is not valid in  {\bf T45m}, whence it is not valid in {\bf S}.
\end{proof}

\section{Final Remarks and Future Work}

Although we have presented an interpretation for Ivlev's semantics, we do not want to argue that the Nmatrices presented here for negation and for modal operators better capture the natural language use of negation, implication and modal operators. There are many possibilities for defining these operators such that all the resulting systems are (conservative) extensions of classical logic. We have presented here only a few of them that already exist in the literature, and that seemed reasonable enough.

From our previous discussion and from the results obtained here, we believe that there are good reasons to deal with modal logics without any introduction rule for modal operators like the necessitation rule. The absence of any modal inference rule  allows us to consider a  radical distinction between a formula being tautology and being necessary, in a similar way that Kripke semantics radically separates the notions between being necessary and being \emph{a priori}.

We also believe that non-deterministic semantics has a promising future in several fields of logic, in particular, modal logic. In the case of propositional modal logic in alethic contexts, they could solve paradoxes that are formulated by the use of the necessitation rule, such as Fitch's paradox, for example. In deontic contexts, systems weaker than {\bf D} could solve deontic paradoxes, like Chisholm's paradox.

There are many technical results to be explored in this field. Lemmon and Scott have proved in \cite{lem:sco:77} a generalized result of completeness through Kripke semantics.\footnote{These results were generalized by H. Sahlqvist in~\cite{Sahlqvist}, by covering a wider class of axiom types.} Let $\Box^n$ and $\Diamond^n$ be denoting $n$-iterations of those operators. Given an axiom schema of the form
$$\Diamond^k \Box^l \alpha \to ^m \Box\Diamond^n \alpha$$
for some fixed natural numbers $k$, $l$, $m$, $n$, it is possible to obtain an accessibility property between the worlds so that the class of models having this property will characterize the system obtained by adding to {\bf K}  the given axiom schema.

An analogous result could be obtained by means of non-deterministic semantics. By varying the number of  truth-values chosen by the multioperator interpreting the  $\Box$ connective (while keeping fixed the multioperators for negation and implication), we have a finite number of submultialgebras between {\bf Tm} and {\bf T45m}. In fact, there are $2 ^4 = 16$ possible multioperators for the modal connective, hence 16 multialgebras can be defined as stipulated above. There are, therefore, 14 systems between {\bf Tm} and {\bf T45m}. The problem here is how to obtain an axiom schema such that, added to {\bf Tm}, produces completeness.
 The same could be done for the other connectives by considering, for example, only those Nnmatrices that extend the classical propositional logic. The previous arguments could also be applied to {\bf Dm} and {\bf Km}.

In another line of research, the generalization of modal Nmatrices to  swap structures semantics proposed in~\cite{CG} opens interesting possibilities for Ivlev-like modal systems with respect to algebraizability in a wide sense (see~\cite{CFG}). Thus, it would be interesting the study of the modal systems proposed here and the ones analyzed in~\cite{con:per:14} from the broader perspective of swap structures.

Finally, another kind of non-deterministic semantics could be applied to Ivlev-like modal systems. For instance, Fidel structures (a.k.a. {\bf F}-structures) could be defined for some of these systems. {\bf F}-structures were originally proposed by M.~Fidel in 1977 for
da Costa's paraconsistent systems $C_n$ (see~\cite{fid:77}), proving for the first time the decidability of such systems. The {\em possible-translations semantics} introduced by W.~Carnielli in~\cite{car:1990} could also we applied to this kind of modal systems. More details about non-deterministic semantics and their relationships can be found in~\cite[Chapter~6]{CC16}. 

We believe that the considerations above show the legitimacy, the importance and the promising future of non-deterministic semantics applied to propositional modal logic. The generalization of this framework to first-order logics is the subject of a forthcoming paper.

\section{Acknowledgements}

This research was partially supported by the Centre International de Math\'ematiques et d'Informatique de Toulouse (CIMI) through contract ANR-11-LABEX-0040-CIMI within the program ANR-11-IDEX-0002-02. Coniglio  acknowledges support from CNPq (Brazil), individual research grant number 308524/2014-4. Peron was partially supported by FAPESC (Brazil), process number FAPESC771/2016.

\bibliographystyle{plain}

\begin{thebibliography}{10}

\bibitem{avr:lev:01}
A.~Avron and I.~Lev.
\newblock Canonical propositional {G}entzen-type systems.
\newblock In {\em Proceedings of the First International Joint Conference on
  Automated Reasoning (IJCAR '01)}, pages 529--544, London, 2001.
  Springer-Verlag.

\bibitem{avr:zam:11}
A.~Avron and A.~Zamansky.
\newblock Non-deterministic semantics for logical systems.
\newblock {\em Handbook of Philosophical Logic}, 16:227--304, 2011.

\bibitem{baaz:96}
M.~Baaz.
\newblock Infinite-valued {G}\"odel logics with 0-1-projections and
  relativizations.
\newblock In P.~H\'ajek, editor, {\em G\"odel '96: Logical foundations of
  mathematics, computer science and physics}, volume~6 of {\em Lecture Notes in
  Logic}, pages 23--33, Berlin, 1996. Springer-Verlag.

\bibitem{car:1990}
W.~A. Carnielli.
\newblock Many-valued logics and plausible reasoning.
\newblock In G.~Epstein, editor, {\em Proceedings of the Twentieth
  International Symposium on Multiple-Valued Logic, Charlotte, NC, USA}, pages
  328--335. The IEEE Computer Society Press, 1990.

\bibitem{CC16}
W.~A. Carnielli and M.~E. Coniglio.
\newblock {\em Paraconsistent Logic: Consistency, contradiction and negation},
  volume~40 of {\em Logic, Epistemology, and the Unity of Science}.
\newblock Springer, 2016.

\bibitem{con:cer:per:15}
M.~E. Coniglio, L.~Fari\~nas~del Cerro, and N.M. Peron.
\newblock Finite non-deterministic semantics for some modal systems.
\newblock {\em Journal of Applied Non-Classical Logics}, 25(1):20--45, 2015.

\bibitem{con:cer:per:17}
M.~E. Coniglio, L.~Fari\~nas~del Cerro, and N.M. Peron.
\newblock Errata and addenda to finite non-deterministic semantics for some
  modal systems.
\newblock {\em Journal of Applied Non-Classical Logics}, 26(1):1--10, 2017.

\bibitem{CFG}
M.~E. Coniglio, A.~Figallo-Orellano, and A.~C. Golzio.
\newblock Non-deterministic algebraization of logics by swap structures.
\newblock {\em Logic Journal of the IGPL}, to appear, 2018.
\newblock Preprint available at {\em arXiv}:1708.08499 [math.LO], 2017.

\bibitem{CG}
M.~E. Coniglio and A.~C. Golzio.
\newblock Swap structures semantics for {I}vlev-like modal logics.
\newblock {\em Soft Computing}, to appear, 2018.

\bibitem{con:per:14}
M.~E. Coniglio and N.~M. Peron.
\newblock Dugundji's theorem revisited.
\newblock {\em Logica Universalis}, 8(3-4):407--422, 2014.

\bibitem{hug:cre:96}
M.~J. Creswell and G.~E. Hughes.
\newblock {\em A New Introduction to Modal Logic}.
\newblock Routledge, London and New York, 1996.

\bibitem{dac:63}
N.~C.~A. da~Costa.
\newblock Sistemas formais inconsistentes ({I}nconsistent formal systems, in
  {P}ortuguese).
\newblock Habilitation thesis, Universidade Federal do Paran\'a, Curitiba,
  Brazil, 1963.
\newblock Republished by Editora UFPR, Curitiba, Brazil,1993.

\bibitem{dem:cla:05}
W.~Demopoulous and Clark P.
\newblock The logicism of {F}rege, {R}ussell and {D}edekind.
\newblock In S.~Shapiro, editor, {\em The Oxford Handbook of Philosophy of
  Mathematics and Logic}, Oxford Handbooks in Philosophy, pages 129--165, USA,
  2005. Oxford University Press.

\bibitem{des:11}
R.~Descartes.
\newblock {\em M\'editations m\'etaphysiques}.
\newblock GF Flammarion, 2011.
\newblock English translation by M. Moriarty in: R. Descartes.
  \emph{Meditations on First Philosophy}. Oxford World’s Classics. Oxford
  University Pres, 2008.

\bibitem{dug:40}
J.~Dugundji.
\newblock Note on a property of matrices for {L}ewis and {L}angford's calculi
  of propositions.
\newblock {\em The Journal of Symbolic Logic}, 5(4):150--151, 1940.

\bibitem{fid:77}
M.~M. Fidel.
\newblock The decidability of the calculi ${C}_n$.
\newblock {\em Reports on Mathematical Logic}, 8:31--40, 1977.

\bibitem{gue:68}
M.~Gueroult.
\newblock {\em Descartes selon l'ordre des raisons}, volume 1. L'ame et Dieu.
\newblock Aubier Montaigne, 1968.
\newblock English translation by R. Ariew in: M. Gueroult. \emph{Descartes’
  Philosophy Interpreted According to the Order of Reasons}, Vol. 1, The Soul
  and God. University of Minnesota Press, 1984.

\bibitem{Hajeck}
P.~H\'ajek.
\newblock {\em Metamathematics of Fuzzy Logic}, volume~4 of {\em Trends in
  Logic}.
\newblock Kluwer, Dordrecht, 1998.

\bibitem{ivl:88}
J.~V. Ivlev.
\newblock A semantics for modal calculi.
\newblock {\em Bulletin of the Section of Logic}, 17(3/4):114--121, 1988.

\bibitem{kant:98}
I.~Kant.
\newblock {\em Critique of Pure Reason (Translated and Edited by Paul Guyer \&
  Allen W. Wood)}.
\newblock The Cambridge Edition of the Works of Immanuel Kant. Cambridge
  University Press, Cambridge, 1998.

\bibitem{kant:92}
R.~Kant, I.;~Meerbote.
\newblock Prolegomena to any future metaphysics that will be able to come
  forward as science.
\newblock In D.~Walford, editor, {\em Theoretical Philosophy, 1755-1770}, The
  Cambridge Edition of the Works of Immanuel Kant, Cambridge, 1992. Cambridge
  University Press.
\newblock Translated by G. Hatfield.

\bibitem{kri:63}
S.~A. Kripke.
\newblock Semantic analysis of modal logic {I}. {N}ormal propositional calculi.
\newblock {\em Zeitschrift f\"ur mathemathische Logik und Grundlagen der
  Mathematik}, 9:67--96, 1963.

\bibitem{kri:66}
S.~A. Kripke.
\newblock Semantical analysis of modal logic {II}. {N}on-normal modal
  propositional calculi.
\newblock In J.~W. Addison, A.~Tarski, and L.~Henkin, editors, {\em The Theory
  of Models}, pages 206--220, Amsterdam, 1966. North-Holland.

\bibitem{kri:81}
S.~A. Kripke.
\newblock {\em Naming and Necessity}.
\newblock Library of Philosophy \& Logic. Blackwell Publishers, 1981.

\bibitem{lem:66}
E.~J. Lemmon.
\newblock Algebraic semantics for modal logics {I}.
\newblock {\em The Journal of Symbolic Logic}, 31(1):44--65, 1966.

\bibitem{lem:sco:77}
E.~J. Lemmon and D.~S. Scott.
\newblock {\em The Lemmon Notes: An Introduction to Modal Logic}.
\newblock Blackwell, 1977.

\bibitem{men:10}
E.~Mendelson.
\newblock {\em Introduction to Mathematical Logic}.
\newblock Discrete Mathematics and Its Applications. Chapman and Hall/CRC, 5
  edition, 2010.

\bibitem{LMont:63}
L.~Monteiro.
\newblock Axiomes ind\'ependants pour les alg\`ebres de {{\L}}ukasiewicz
  trivalentes.
\newblock {\em Bulletin de la Societ\'e des Sciences Math\'ematiques et
  Physiques de la R. P. Roumanie, Nouvelle s\'erie}, 7:199--202, 1963.

\bibitem{noo:13}
H.~Noonan.
\newblock {\em Routledge Philosophy GuideBook to Kripke and Naming and
  Necessity}.
\newblock Routledge Philosophy GuideBooks. Routledge, 2013.

\bibitem{omo:sku:16}
H.~Omori and D.~Skurt.
\newblock More modal semantics without possible worlds.
\newblock {\em IfCoLog Journal of Logics and their Applications},
  3(5):815--846, 2016.

\bibitem{russ:10}
B.~Russell.
\newblock {\em The Philosophy of Logical Atomism (Routledge Classics)}.
\newblock Routledge, 2010.

\bibitem{Sahlqvist}
H.~Sahlqvist.
\newblock Completeness and correspondence in first and second order semantics
  for modal logic.
\newblock In S.~Kanger, editor, {\em Proceedings of the Third Scandinavian
  Logic Symposium}, pages 110--143, Amsterdam, 1975. North Holland.

\bibitem{seg:71}
K.~K. Segerberg.
\newblock {\em An essay in Classical Modal Logic}.
\newblock PhD thesis, Stanford University, Stanford, 1971.

\bibitem{woj:84}
R.~W\'{o}jcicki.
\newblock {\em Lectures on propositional calculi}.
\newblock Ossolineum, Wroclaw, Poland, 1984.

\end{thebibliography}

\end{document}